\documentclass[reprint,superscriptaddress,prl]{revtex4-2}

\usepackage{amsmath}
\usepackage{amssymb}
\usepackage{amsthm}
\usepackage{graphicx}
\usepackage{dcolumn}
\usepackage{bm}
\usepackage{nccmath}
\usepackage{mathrsfs}
\usepackage{mathtools}
\usepackage{color}
\usepackage{transparent}
\usepackage{comment}
\usepackage{here}

\setcitestyle{super,sort&compress}
\usepackage[hidelinks]{hyperref}
\theoremstyle{plain}
\newtheorem{dfn}{Definition}

\newtheorem{thm}[dfn]{Theorem}

\newcommand{\R}{\mathbb{R}}
\newcommand{\N}{\mathbb{N}}

\newlength{\mathindent}
\newtheorem{supptheorem}{Theorem}[section]

\newtheorem{suppcor}[supptheorem]{Corollary}
\newtheorem{supprem}[supptheorem]{Remark}

\newtheorem{suppdfn}[supptheorem]{Definition}

\allowdisplaybreaks[4]

\begin{document}

\preprint{APS/123-QED}

\title{Ensemble Reservoir Computing for Physical Systems}

\author{Yuma Nakamura}
\affiliation{Division of Mathematical and Physical Sciences, Kanazawa University, Japan}
\email{nakamurayuma@stu.kanazawa-u.ac.jp}
\author{Tomoyuki Kubota}
\affiliation{Graduate School of Information Science and Technology, The University of Tokyo, Japan}%
\affiliation{Next Generation Artificial Intelligence Research Center (AI Center), The University of Tokyo, Japan}
\author{Yusuke Imai}
\affiliation{Graduate School of Information Science and Technology, The University of Tokyo, Japan}%
\author{Sumito Tsunegi}
\affiliation{National Institute of Advanced Industrial Science and Technology, Japan}
\author{Hirofumi Notsu}
\affiliation{Faculty of Mathematics and Physics, Kanazawa University, Japan}
\author{Kohei Nakajima}
\affiliation{Graduate School of Information Science and Technology, The University of Tokyo, Japan}
\affiliation{Next Generation Artificial Intelligence Research Center (AI Center), The University of Tokyo, Japan}

\begin{abstract}
  Physical computing exploits unconventional physical substrates to overcome limitations such as the high energy consumption inherent in digital computation.
  However, intrinsic noise and temporal fluctuations (e.g., oscillations) generally deteriorate computational performance.
  Here, we propose ensemble reservoir computing (ERC), a novel framework that employs ensemble averaging of spatially multiplexed systems to achieve robust information processing despite noise and temporal fluctuations.
  First, we prove that ensemble averaging in ERC eliminates temporal fluctuations and noise from dynamical states under certain conditions, thereby restoring computational performance to its noise-free level.
  Next, we show that ERC not only removes the noise and fluctuations but also actively exploits the computational capabilities that conventional reservoir computing (RC) leaves unutilized.
  This computational enhancement is demonstrated across diverse dynamical systems (e.g., periodic, chaotic, and strange-nonchaotic systems), in which ERC outperforms conventional RC.
  Finally, using energy-efficient spin-torque oscillators (STOs), we demonstrate that ERC maintains high performance even under realistic conditions, in which noise and temporal fluctuations coexist: STOs with ERC achieved 99\% accuracy on an error detection test, where conventional STO reservoir with linear regression only shows a chance level performance, highlighting ERC's robustness and performance gains for physical systems.
\end{abstract}

\maketitle
\section{Introduction}
The development of computing devices using unconventional materials has emerged as a key strategy to overcome the limitations of conventional computers, particularly their high energy consumption. 
This challenge has become especially pressing in the context of artificial intelligence~\cite{Momeni2025TrainingPNN}.
The rapid increase in computational demand has led to estimates that, by 2026, global AI power consumption will rival the total electricity usage of Japan~\cite{Energy_comsumption}. Among promising alternatives, physical computing substrates such as spin-torque oscillator (STO)-based devices exhibit substantially higher energy efficiency than semiconductor technology~\cite{Nature_quantum}.
Computing devices built from these materials can operate at only $0.1$\% of conventional computer power consumption~\cite{materials}; however, they inherently exhibit noise and system-specific variability (e.g., nonstationarity and periodicity), making the management of these disturbances essential for reliable operation.
\par
To date, the interfering factors in computational systems have been managed in two ways.
First, error management techniques have been developed to handle noise.
For example, in digital and quantum computers, errors are corrected by majority voting over multiple signal copies (a scheme known as the repetition code~\cite{fudamental_err_code,quantum_computer,quantum_theory}), which improves robustness against noise.
Second, temporal fluctuations, which reflect time-dependent system dynamics, are eliminated to achieve stable outputs.
For example, digital computers suppress transistor transient responses by waiting briefly before sampling, thereby obtaining reliable binary results~\cite{harris2021digital}.
In some control applications, angular velocity is estimated through time integration of angular acceleration, which causes errors to accumulate.
To counteract this nonstationary increase, an error model is estimated and removed from the measurement.
Moreover, neuromorphic oscillators (e.g., STOs) retain intrinsic oscillations even under external input, oscillations that are removed using postprocessing techniques such as envelope extraction~\cite{Huang1998}.
However, because these fluctuations may encode processed input, they could potentially be harnessed to improve the computational capabilities of these systems~\cite{GRC,TIPC,Kubota_ipc}.
\par
Physical reservoir computing (PRC) is an excellent approach for building physical computing systems using unconventional materials~\cite{materials}. In PRC, we apply input to a physical system and attach a readout, allowing the framework to be implemented across a range of physical substrates (e.g., quantum, optical, electronic, and nanomaterial systems~\cite{Fundamental_Reservoir,Spin_Reservoir,Nature_quantum,Tsunegi_2023,Akashi_2022,Akashi_2020, Yamaguchi_2020,Yamaguchi_2020_jun,Tsunegi_2019,Tsunegi2018MemoryCapacitySTO,Taniguchi2021ReservoirSpintronics}). 
A problem is that the conventional PRC framework cannot
remove the effects of noise and system fluctuations, which prevent stable output production and a full use of their computational capabilities.
To overcome this problem, one solution is to introduce noise reduction transformation to the readout layer~\cite{GRC}, which can successfully remove the negative effects of the system state to utilize the latent computational capabilities of the hopeful devices.
However, since this transformation is not universal across systems, the design of the readout function remains a significant issue.
\par
\begin{figure*}
  \includegraphics[width=\textwidth]{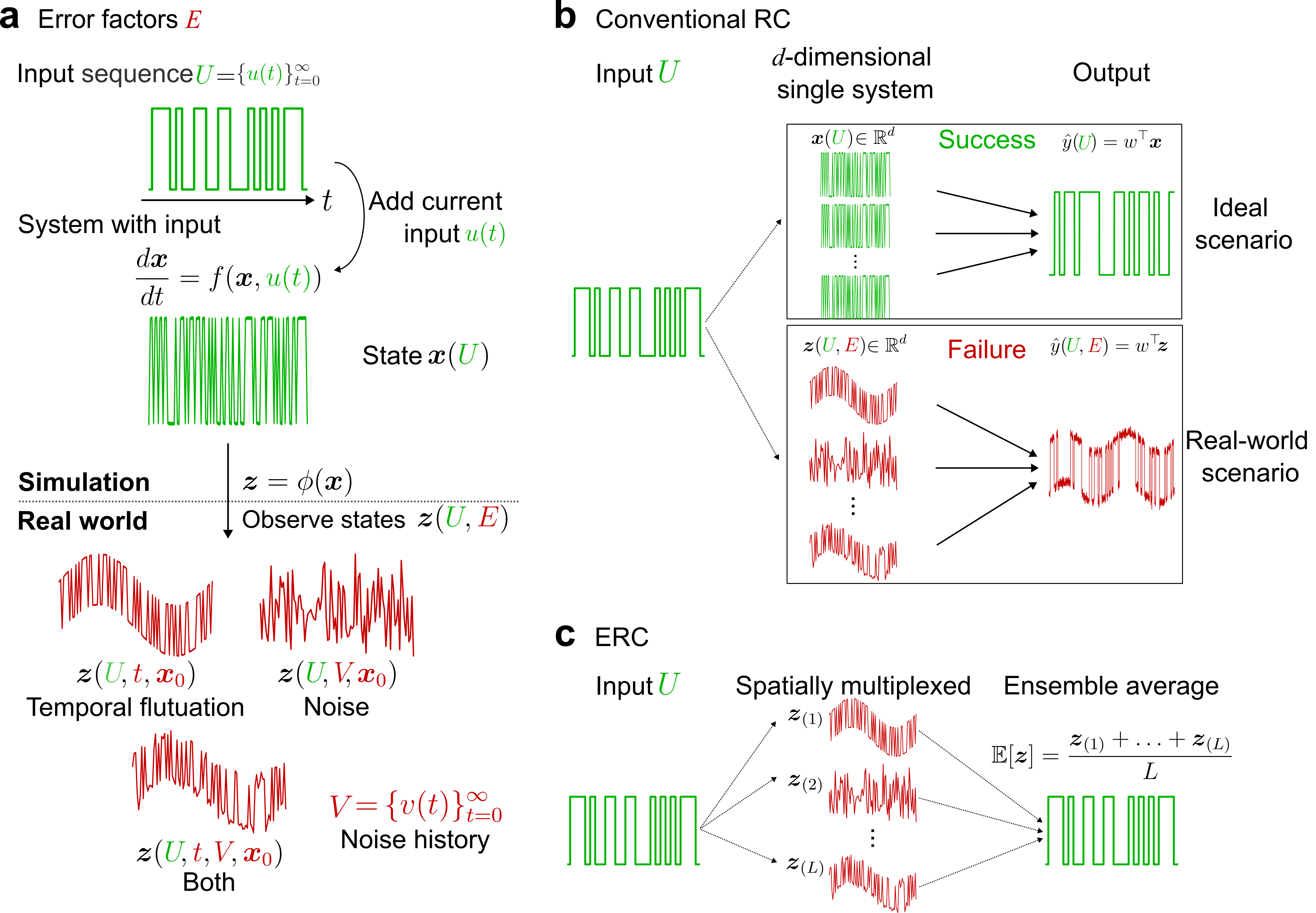}
  \caption{
  \textbf{Schematics of ensemble reservoir computing (ERC).}
  \textbf{a}) Examples of error factors $E$ in a physical system.
  In simulations, the physical reservoir state $\boldsymbol{x}$ treated as an ideal function without error factors $E$ (green).
  In real-world implementations, the dynamical state is observed as the measured state $\boldsymbol{z} = \phi(\boldsymbol{x})$, a nonlinear transformation of the original state $\boldsymbol{x}$ through $\phi$ that contains error factors $E$ (i.e., noise and temporal fluctuations; red).
  The green and red lines represent the ideal physical system and the observed state, respectively.
  \textbf{b}) Illustration of an ideal conventional RC scenario (upper) versus its practical limitations in the real world (lower).
  Conventional RC makes predictions using a linear sum of states that depends solely on input history.
  However, in practice, the error factors $E$ distort the observed states in physical systems.
  \textbf{c}) The operating principle of ERC. Parallel identical systems driven by a common input yield multiple observed states $\boldsymbol{z}_{(i)}~(i=1,2,\ldots,L)$.
  The ensemble average $\mathbb{E}[\phi(\boldsymbol{x})]$ over the observed states eliminates the error factors (i.e., noise and temporal fluctuations).
  }
  \label{fig:schematics}
\end{figure*}
In this paper, we propose ensemble reservoir computing (ERC), a novel framework that eliminates noise and fluctuations by ensemble averaging over parallel systems instead of system-specific nonlinear transformations, and we show that it enables the construction of physical computing systems.
As a result, ERC not only removes noise but also extracts latent computational capability from temporal fluctuations. We theoretically prove that, under certain assumptions, ERC completely eliminates dependence on noise and temporal fluctuations.
Furthermore, we demonstrate its effectiveness in systems with noisy and fluctuating dynamics (i.e., chaotic systems and STO-based devices), which conventional PRC cannot reliably exploit. For these reasons, ERC paves the way for reliable and high-performance computing in an unprecedented range of real-world, noisy, and temporally complex systems.
\section{Results}
\subsection{Ensemble reservoir computing}
We first introduce the dynamical system that forms the basis of the ERC framework.
We describe a state equation of the dynamical system that is driven by input $u(t)$, as follows:
\begin{equation*}
  \frac{d\boldsymbol{x}}{dt} = f(\boldsymbol{x}, u(t)),
\end{equation*}
where $\boldsymbol{x}$ denotes the $d$-dimensional system state.
In practice, not all components of $\boldsymbol{x}$ can be observed directly; instead, a subset is observed as $\boldsymbol{z}$, which may be nonlinearly transformed and is defined by the following observation equation:
\begin{equation*}
  \boldsymbol{z}=\phi(\boldsymbol{x}),
\end{equation*}
where $\phi$ is the observable function.
Temporal fluctuations can arise through the time evolution of the state equation (e.g., limit cycle and chaos).
As shown in Fig.~\ref{fig:schematics}a, a physical system can be regarded as receiving both dynamical and observational noise.
\par
Such fluctuations and noise cannot be ignored in the construction of a computational system.
In this study, we specifically consider dependencies on noise, time, and initial values and denote them collectively by $E$. 
Consequently, $\boldsymbol{x}$ and $\boldsymbol{z}$ can be regarded as functions of $(U, E)$, where $U:=\{u(t)\}_{t\geq 0}$ denote the input history. We call the system {\it time-variant} when the state retains factors of $E$; otherwise, it is {\it time-invariant}. From now on, the term temporal fluctuations or simply fluctuations is used to express the time-variant factor intrinsic to the system, which is different from external noise and independent of input.
\par
In conventional reservoir computing (RC), the output is obtained through a linear summation, which serves as the theoretical foundation for PRC; therefore, the time-variant state cannot produce the appropriate output (Fig.~\ref{fig:schematics}b).
Generalized reservoir computing (GRC)~\cite{GRC} mitigates this problem by converting the time-variant states into time-invariant through a nonlinear transformation inserted between the dynamical state and the linear readout.
Note that this nonlinear transformation can be as simple as a second-order polynomial, although its effectiveness depends on type of system. To determine the transformation, GRC with a multi-layer perceptron (MLP) readout (MLP-GRC) was previously adopted as an example, as it is outstanding at learning input-output relations~\cite{Rumelhart1986,Goodfellow2016DeepLearning}.
\par
To overcome the influence of error factors $E$ and effectively extract computational capability from contaminated dynamics, we incorporate ensemble averaging~\cite{Thomas2016Ensemble} into RC.
We prepare a set of parallel identical systems receiving common inputs~\cite{PhysRevApplied.11.034021}, referred to as spatially multiplexed systems, and take their weighted ensemble average $\mathbb{E}[\phi(x_k)]$, where $\phi(x_k)$ represents the $k$-th component of the observed state (Fig.~\ref{fig:schematics}c).
The output $\hat{y}_t$ is then computed by a linear summation of the weighted ensemble-averaged state, as follows:
\begin{equation*}
  \hat{y}_t:=\sum_{k=1}^{d}\alpha_k\mathbb{E}[\phi(x_k)],
\end{equation*}
where $\{\alpha_k\}_{k=1}^{d}$ is obtained by minimizing the error between the target signal $y_t$ and $\hat{y}_t$.
We call this framework ERC.
\par
Next, we prove that $\mathbb{E}[\phi(x_k)]$ eliminates noise and temporal fluctuations. The $l$-th system evolves independently with its own initial value $\boldsymbol{x}_0^{(l)}$ for $l = 1, 2, \ldots$.
Here, the superscript $l$ denotes the system index, and each $\boldsymbol{x}_0 \in\mathbb{R}^d$ represents a $d$-dimensional initial state.
We define the ensemble average of the observed state $\phi(x_k)$ for identical systems differing only in their initial states as
\begin{equation}
  \mathbb{E}[\phi(x_k)] \coloneqq \lim_{L \to \infty}\sum_{l=1}^L \omega_l^{(L)} \phi \left( x_k(U, E^{(l)}) \right),
  \label{eq:initial_value_average}
\end{equation}
where $\{\omega_l^{(L)}\}_{l=1}^{L}$ denote the weights.
In this paper, this operation is termed the weighted ensemble average transformation.
The following theorem establishes that this transformation eliminates explicit dependence on the error terms and yields a time-invariant quantity.
\begin{thm}\label{thm:thm1}
Let $\{\omega_l^{(L)}\}_{l=1}^{L}$ be weights that satisfy certain conditions {\rm (}where $\omega_l^{(L)} = 1/L$ for $l = 1,\ldots, L$ is a typical choice{\rm )}. Then, the ensemble-averaged quantity $\mathbb{E}[\phi(x_k)]$ is time-invariant in the following two cases:
\begin{itemize}
  \item[{\rm(}i{\rm)}] the error factor is only i.i.d. noise,
  \item[{\rm(}ii{\rm)}]
  the error factor represents temporal fluctuations, and $x_k$ admits a Fourier series representation with input-dependent coefficients.
  In this case, $\phi$ is assumed to be expandable in a Fourier or power series.
  \end{itemize}
  {\rm (See Theorem S.~1 and Corollary S.~2 in the Supplementary Information for further details; proof strategies for each case are outlined in Cases I and II below.)
  In this way, ERC eliminates the influence of the error factors $E\colon$ initial value dependence, temporal fluctuations, and noise.
  Intuitively, the parallel systems share the input but contain independent signals such as noise and/or part of fluctuations.
  Theorem 1 explains that the ensemble average removes the independent signals and thus makes functions of only input history.
  For practical implementation, we use the typical weights $\omega_l^{(L)} = 1/L$, $l = 1, \ldots, L$, and refer to the weighted ensemble average transformation simply as the ensemble average transformation.
  }
\end{thm}
\begin{figure*}[t]
  \centering
  \includegraphics[width=\textwidth]{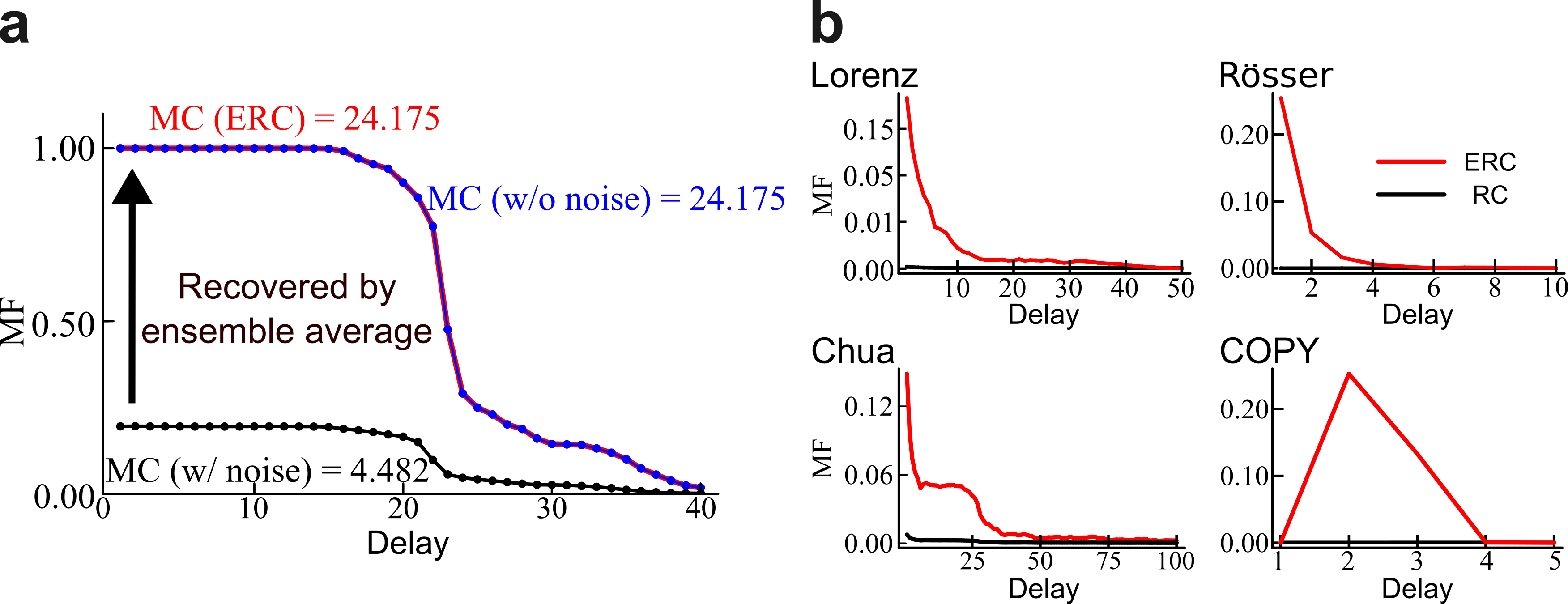}
  \caption{\textbf{Ensemble averaging restores and improves memory in time-variant systems.}
  \textbf{a}) An example of noise removal using an echo state network (ESN).
  Three memory functions (MFs) and their memory capacities (MCs) are shown: a standard ESN (blue), an ESN with noise (black), and an ESN with ERC (red).
  To numerically verify Theorem \ref{thm:thm1}, we used $L=5\times10^5$ parallel ESNs with a network size $d=30$ and spectral radius $\rho=0.94$.
  \textbf{b}) Four examples that exploit temporal fluctuations: the Lorenz system, R\"{o}ssler system, Chua circuit, and COPY map.
  The MFs obtained using conventional RC (black) and ERC (red) are illustrated.
  Their MCs under ERC are $1.08$ (Lorenz), $0.7$ (R\"{o}ssler), $1.57$ (Chua), and $0.37$ (COPY).
  }
  \label{fig:MC_ERC's}
\end{figure*}
\vspace{-1.5\baselineskip}
\subsubsection{Case I : Noise tolerance}
First, to examine the noise tolerance, we consider a system driven by input $u_m$ and noise $v_m$, where the noise is independent of time and initial values.
We assume that the realizations $v^{(l)}_m \in [a,b]$ vary across trials.
Then, by the law of large numbers, $\mathbb{E}[\phi(x_k)]$ in Eq.~\eqref{eq:initial_value_average} is equivalent to integrating over $v\in[a,b]$, which is proven to eliminate the influence of noise (see Theorem S.~1 in the Supplementary Information).
\par
To numerically verify the noise tolerance of the ensemble average transformation, we use an echo state network (ESN)~\cite{ESN_jeager} with noise, defined as follows:
\begin{equation}
  \boldsymbol{r}_{m+1} = \tanh\left(W\boldsymbol{r}_m +\alpha W_{\rm in} u_m + \sigma W_{\rm in} v_m^{(l)}\right),
  \label{eq:noisy_ESN}
\end{equation}where $\boldsymbol{r}_m\in\mathbb{R}^{d}$ is the ESN state at time step $m$, $W \in \mathbb{R}^{d \times d}$ is a fixed recurrent weight matrix with spectral radius $\rho=0.94$, and $W_{\rm in} \in \mathbb{R}^{d}$ is a fixed input weight vector.
The input $u_m$ is sampled uniformly from $[0,1]$, and the noise $v^{(l)}_m$ from $[-1,1]$.
In the absence of noise $(\sigma=0)$, the state $\boldsymbol{r}_m$ becomes time-invariant~\cite{ESN_jeager}.
To evaluate the removal of noise in ERC, we use the memory function (MF), which measures the contribution of past inputs $u_{m-\tau}$ to the system state.
We denote the MF at delay $\tau$ by $M_{\tau}$ and define the memory capacity (MC)~\cite{MC_jeager} as $\sum_{\tau} M_{\tau}$ ( see Methods for further details).
 Figure~\ref{fig:MC_ERC's}a compares the MF of a standard ESN (blue), an ESN with noise (black), and an ESN with ERC (red).
 To focus on the ensemble average transformation, without nonlinear processing, we use the linear observation function $\phi(x) = x$.
In the absence of noise ($\sigma=0$), the MC is close to $24$. With noise ($\sigma\neq0$), the MC drastically decreases by 81\%, suggesting that noise significantly degrades the system's computational capability.
However, the ensemble average transformation restores the MC to its original level, even when the system is driven by both input and noise.
This result indicates that ERC can remove noise from contaminated dynamical states and thereby greatly improve computational performance.
\subsubsection{Case II : Temporal-fluctuation tolerance}
Another adverse factor is temporal fluctuation arising from time-dependent dynamics, such as periodic or chaotic behavior.
To examine tolerance to such fluctuations, we describe the system state using a Fourier series with coefficient functions that depend on input history.
Here, $\{f_j(t,U)\}_{j=1}^{N_{*}}$ are general functions of time $t$ and input history $U$, while the random phases $\theta_j$, drawn independently from $[-\pi,\pi]$, represent the dependence on the initial conditions.
We then assume that the state $x$ is represented by the following finite series:
\begin{equation}
  \scalebox{0.95}{$
  \displaystyle
  x(t,U,\boldsymbol{x}^{(l)}_0)\simeq
  \sum_{j=1}^{N_{*}}\left[a_j\cos(f_j+\theta^{(l)}_j)+
  b_j\sin(f_j+\theta^{(l)}_j)\right]+
  C,
  $}
  \label{eq:basic_trigometric}
\end{equation}
where $a_j=a_j(U)$, $b_j=b_j(U)$, and $C=C(U)$ depend only on the input history.
The ensemble average transformation eliminates the time-dependent factors $\{f_j\}_{j=1}^{N_{\ast}}$ because it corresponds to an integration over all phases $\theta_1, \ldots, \theta_{N_{\ast}}$; for instance, $\int_{-\pi}^\pi \cos^k(\theta+\eta)\, d\theta$ yields a constant for $\eta=\eta(t,U)$ and $k\in\N$.
As a result, the transformed state $\mathbb{E}[\phi(x)]$ becomes independent of $t$ and $\{\theta_j\}_{j=1}^{N_{\ast}}$, making it useful for RC.
\par
To demonstrate the effectiveness of the ensemble average transformation, we examine it analytically in a model system. We consider the Stuart–Landau system with input applied in the radial direction, described as follows:
\begin{figure*}[t]
  \includegraphics[width=\textwidth]{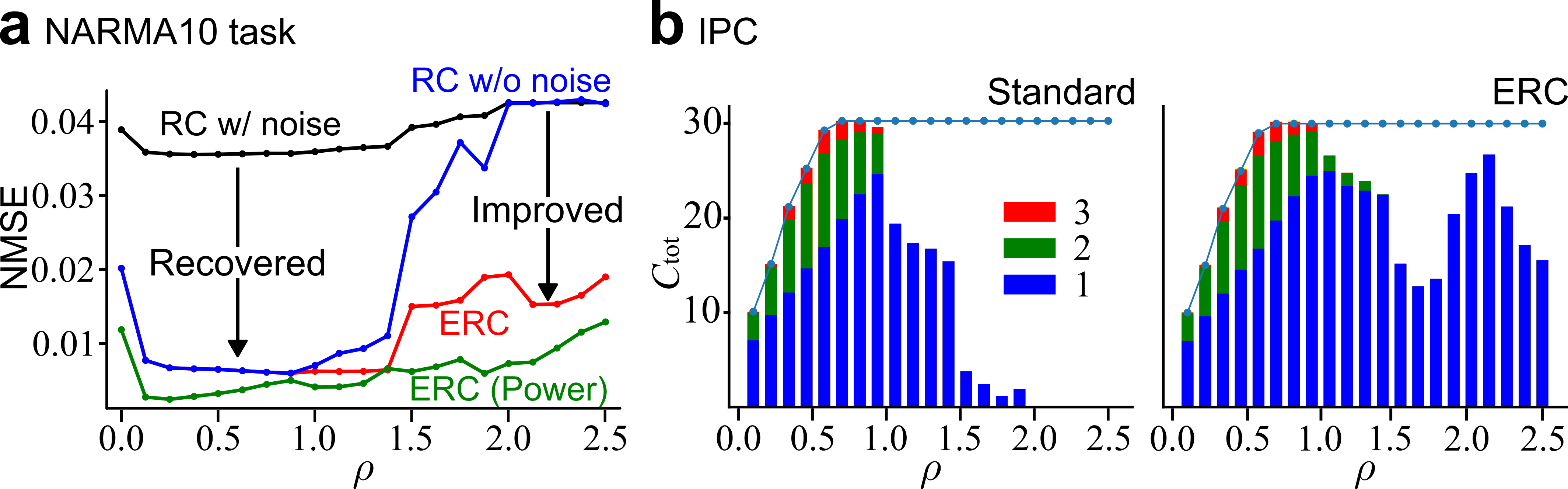}
  \caption{
  \textbf{Application of ERC to a} time-variant ESN.
  \textbf{a}) Prediction errors (NMSE) on the NARMA10 task as a function of $\rho$ for four configurations: a standard ESN without noise (blue), an ESN with additive noise (black), ERC with first-order ensemble average (red), and ERC with ensemble averages of powers $1$–$10$ (green).
  \textbf{b}) Information processing capacity (IPC) as a function of $\rho$ for the standard ESN (left) and ERC (right). The bars indicate the sum of $d$th-order IPCs (1st [blue], 2nd [green], and 3rd [red]), while the lines indicate the state rank.
  }
  \label{fig:ESN_robust}
\end{figure*}
\begin{equation}
  \begin{split}
    &\dot{x}=\alpha x-\beta y-x(x^2+y^2)+\sigma \frac{x}{\sqrt{x^2+y^2}}u,\\
    &\dot{y}=\beta x+\alpha y-y(x^2+y^2)+\sigma \frac{y}{\sqrt{x^2+y^2}}u,\\
  \end{split}
  \label{eq:radiue_input_hopf_inP15}
\end{equation}where $\alpha > 0$ represents the radius of the limit cycle and $\beta \in \mathbb{R}$ is the natural frequency of the system.
The dependence on the initial value is expressed as an initial phase $\theta$, and the solutions of Eq.~\eqref{eq:radiue_input_hopf_inP15} can be written as $
(x,y) = \left( r(U) \cos(\beta t + \theta),\ r(U) \sin(\beta t + \theta) \right)$, which satisfy the assumption of Eq.~\eqref{eq:basic_trigometric} with $N_{*}=1$ ($a_1=b_1=r(U)$, $f_1=\beta t$, and $C=0$).
The corresponding ensemble averages with $\phi(x)=x^n$ are derived as follows:
\begin{equation}
  \scalebox{0.99}{$
  \mathbb{E}[x^n]=\mathbb{E}[y^n]=
  \left\{
  \begin{array}{cc}
  \displaystyle 0 & (n=1,3,5,\ldots)\\
  \displaystyle \binom{n}{n/2}\left(\frac{r}{2}\right)^{n} & (n=2,4,6,\ldots)
  \end{array}
  \right.,
  $}
  \label{eq:phi_SLsystem}
\end{equation}
where $r = r(U)$ depends only on the input history, indicating that $\mathbb{E}[x^n]$ and $\mathbb{E}[y^n]$ are time-invariant for all $n\in \mathbb{N}$.
This result shows that although the original solution contains only time-varying quantities, the ensemble average transformation extracts time-invariant terms, demonstrating that systems with temporal fluctuations can serve as computational resources.
\par
 We identify two limitations of ERC.
First, ERC is effective when there is no spatial synchronization among the states of spatially multiplexed systems. Conventional RC requires dynamical states to synchronize with the input to yield a reproducible output, which is opposite to the requirement for ERC (see Supplementary Information for further details). 
Second, we introduce the information processing capacity (IPC) to characterize the input information embedded in time-invariant quantities (see Methods for further details). The IPC extends the MC by incorporating nonlinear dependencies and is quantified using coefficients from an orthogonal-polynomial expansion of the input history. For high-dimensional states or large $N_{\ast}$, the corresponding bases can be identified individually, and once identified, any function expressed in this basis can also be determined (see Supplementary Information for further details).
\subsection{Extension of the admissible dynamics and observation functions}
\begin{figure*}[t]
  \includegraphics[width=\textwidth]{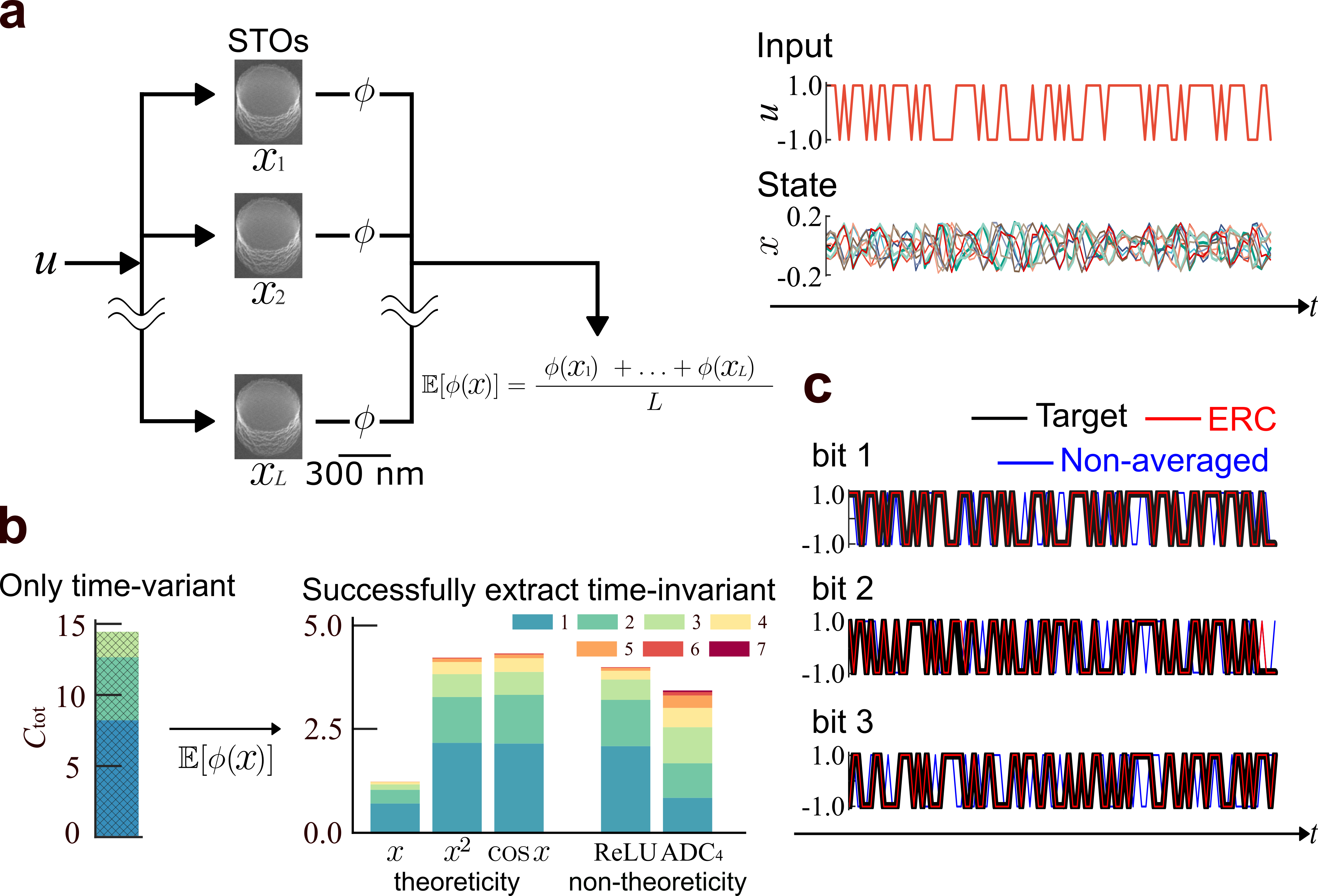}
  \caption{
  \textbf{Application of ERC to a} real spin-torque oscillator (STO).
  \textbf{a}) Schematic illustration of ERC applied to $L=60$ STOs (left).
  The spatially multiplexed STOs are illustrated using SEM images of the resist-pillar patterns before ion milling; these pillars serve as etching masks for the STO nanopillars.
  We inject the same binary input signal (upper right) to into all STOs and observe their resulting time series (lower right).
  \textbf{b}) The temporal information processing capacity (TIPC) of the original STO waveform and the IPC of the ensemble-averaged states using the observation functions $\phi\in\{x,x^2,\cos x,\mathrm{ReLU}(x),\mathrm{ADC}_{4}(x)\}$. 
  The TIPC of the original waveform includes only time-variant capacities. 
  In contrast, transformed states with several $\phi$ contain time-invariant components. Here, $x,x^2$, and $\cos x$ are theoretically admissible, while $\mathrm{ReLU}(x)$ and $\mathrm{ADC}_{4}(x)$ are not theoretically guaranteed.
  The color of each bar indicates the input order. The hatched and unhatched bars denote time-variant and time-invariant components, respectively.
  \textbf{c}) Time-series comparison among the target (black), non-averaged (blue), and ERC (red) signals for a cyclic redundancy check (CRC) task.
  We use $\phi\in\{x^2,x^4,x^6,x^8\}$ as observation functions to solve the CRC task.}
  \label{fig:equippment_data}
\end{figure*}
Next, to broaden the applicability of ERC, we consider
two cases: (i) system states cannot be represented by the series expansion in Eq.~\eqref{eq:basic_trigometric} and (ii) systems with both time and noise dependencies.
As an example of case (i), we consider systems with a fractal structure, including three chaotic systems (the Lorenz, R\"{o}ssler, and Chua models), and a strange-nonchaotic system the COPY map~\cite{SNA} ( see the Models subsection of Methods for details).
It is unclear whether any of these systems can be expressed in the form of Eq.~\eqref{eq:basic_trigometric}.
We examined the ability of ERC to extract latent computational capabilities from such systems by evaluating their MFs.
As shown in Fig.~\ref{fig:MC_ERC's}b, although the MFs of the original states provide values only on the order of $10^{-4}$ (black), the ensemble average transformation recovers meaningful memory (red).
Using ERC, we obtained MCs of 1.08 (Lorenz), 0.70 (R\"{o}ssler), 1.57 (Chua), and 0.37 (COPY), where the nonlinear functions used were $\phi(x)=x$, $\phi(x)=\log(x^2)$, $\phi(x)=x$, and $\phi(x)=\exp(x^2)$, respectively.
For comparison, the Lorenz and R\"{o}ssler systems are reported as application examples of MLP-GRC~\cite{GRC}, with MCs of 0.97 for both.
These results indicate that ERC can extract information from time-variant systems without requiring complex nonlinear transformations, such as MLPs, and that it remains effective even without satisfying Eq.~\eqref{eq:basic_trigometric}.
\par
As examples of case (ii), we consider the noisy ESN defined in Eq.~\eqref{eq:noisy_ESN}.
We provide an overview of the dynamical properties of the ESN to show that ERC can leverage dynamics in the time-variant regime.
We identify the parameters that govern the time-variant and the time-invariant regimes in the ESN by sweeping the spectral radius $\rho$ of $W$ to span different dynamical regimes.
The ESN without input exhibits fixed-point dynamics for $\rho < 1$, a periodic window around $\rho \approx 1$, and chaotic behavior for $\rho > 2$. Furthermore, including external input and noise leaves this bifurcation structure almost unchanged (see Fig.~S2 in the Supplementary Information for further details).
Thus, the ESN is time-invariant for $\rho < 1$ and time-variant for $\rho > 1$.
\par
To demonstrate the effectiveness of ERC in the noisy time-variant regime, we evaluated its performance using the standard nonlinear auto-regressive moving average of order
10 (NARMA10) benchmark and the IPC metric (see Methods for details).
 Figure~\ref{fig:ESN_robust}a shows the NARMA10 results for a standard ESN (blue), a noisy ESN (black), ERC with $\phi(x)=x$ (red), and ERC with $\phi\in \{x,\ldots,x^{10}\}$ (green).
In the time-invariant regime ($\rho<1$), the standard ESN shows high performance.
 In contrast, performance deteriorates sharply in the time-variant regime ($\rho>1$).
After noise injection, performance drops for all values of $\rho$.
However, ERC fully restores accuracy lost to noise in the $\rho<1$ regime and consequently outperforms the standard ESN when $\rho>1$.
In the time-invariant regime ($\rho<1$), ERC matches the IPC of the standard ESN.
In the time-variant regime, ERC extracts more IPC than the standard ESN, suggesting that ensemble averaging converts time-variant terms into time-invariant components (Fig.~\ref{fig:ESN_robust}b).
Moreover, applying different nonlinear transformations expands the set of reservoir states, leading to improved performance.
These results indicate that ERC provides a flexible framework for implementing PRC across a wide range of physical systems.
\subsection{Physical implementations}
Finally, to demonstrate applicability of ERC to physical systems, we adopted STOs, since the Stuart--Landau model in Eq.~\eqref{eq:radiue_input_hopf_inP15} is an approximation of STO dynamics and has already demonstrated the effectiveness of ERC.
Figure~\ref{fig:equippment_data}{a}(left) illustrates the ERC procedure applied to parallel STO systems driven by binary random inputs, which were applied in the manner of two-level modulation with 150 or 300mV (see Methods for further details).
The STOs do not exhibit spatial synchronization across trials (Fig.~\ref{fig:equippment_data}a, right). This independence ensures that each oscillator provides distinct dynamics suitable for ERC.
\par
First, we evaluate ERC with the real STOs using the temporal IPC (TIPC)~\cite{Kubota_ipc,TIPC}, which generalizes the IPC by incorporating temporal dependencies so that processed inputs can be resolved not only in time-invariant states but also in time-variant states unusable in conventional RC (see Methods for further details). In this evaluation, we used not only theoretically guaranteed transformations (i.e., those whose $\phi$ can be expanded in Fourier or power series [see Theorem~\ref{thm:thm1}]) but also those without such guarantees.
Figure~\ref{fig:equippment_data}{b} (left) presents the TIPC of the original STO waveform, which contains only time-variant capacities, indicating that conventional RC cannot form a time-invariant output.
Figure~\ref{fig:equippment_data}{b} (right) illustrates only the time-invariant capacities obtained by ensemble averaging with several nonlinear functions $\phi$.
Specifically, we consider $\phi\in\{x, x^2, \cos x, \mathrm{ReLU}(x), \mathrm{ADC}_{4}(x)\}$, where $\mathrm{ReLU}(x)\coloneqq \max(0,x)$ and $\mathrm{ADC}_{4}(x)$ denotes a $4$-bit analog-to-digital transformation~\cite{max_Nakajima} (see Methods, Observation Functions, for details).
The figure also compares results obtained using theoretically admissible nonlinear transformations $\phi$ ($x$, $x^2$, $\cos x$) (the three bars on the left) and transformations not guaranteed by Theorem 1 [ReLU($x$), ADC${}_4$($x$)] (the two bars on the right).
In all cases, the time-invariant capacities are successfully recovered from the STO time series that contain only time-variant terms.
\par
Next, we evaluate the performance of ERC on a $3$-bit cyclic redundancy check (CRC) task used for error detection in signal processing (see Methods).
Figure~\ref{fig:equippment_data} c compares the target (black), non-averaged (blue), ERC-predicted (red) time-series signals for each bit.
The task is solved by ensemble averaging with nonlinear transformations $\phi \in \{x^2, x^4, x^6, x^8\}$, while the non-averaged state is used only as a transformed value. In ERC, nonlinear transformations and their ensemble averages serve as reservoir states—specifically, $\mathbb{E}[x^2]$, $\mathbb{E}[x^4]$, $\mathbb{E}[x^6]$, and $\mathbb{E}[x^8]$.
The task performance is $67.3$\% and $99.4$\% for the non-averaged and ERC cases, respectively.
Furthermore, we performed the CRC task using ERC with 2-, 3-, 4-, and 5-bit ADC($x$) as nonlinear functions, taking practical implementation into account, and achieved $94.8\%$ (see the Supplementary Information for details). ERC with STOs is further applied to Hamming code encoder and decoder error-correction tasks (see Methods).
In the encoder task, a $7$-bit signal is emulated from a $3$-bit input sequence by adding information to correct a bit-flip error. In the decoder task, a corrected $3$-bit sequence is emulated from the received $7$-bit sequence. Using the observation functions $\phi\in\{x^2,x^4,x^6,x^8\}$, we achieved $99.9\%$ accuracy for the encoder task and $89.4\%$ accuracy for the decoder task.
In addition, we compare the IPC and MC of the reservoir states used for task execution in this work with those reported in previous STO-based reservoir computing studies (see Table~S2).
 The results obtained in this work exhibit the highest MC and IPC among physically implemented STO-based reservoirs.
In the simulation work, there is a setup that shows higher MC in Ref.~\cite{Akashi_2020}, which utilizes the effect of spatial coupling. In spintronics, aside from the STO-based reservoir computing platform (e.g., spin-wave and spin-ice), similar results have been reported\cite{Namiki2023,Namiki2025}, which could be explored in our future work.
These results indicate that ERC extends beyond its theoretical domain, establishing a general computational paradigm that effectively utilizes the time-variant system as reservoirs.
\section{Discussion}
In this paper, we applied ERC to various types of systems, from which we derived three simple design rules for practical use:
(i) avoid synchronization among spatially multiplexed systems to preserve phase diversity,
(ii) employ physical systems that allow sufficient data acquisition for ensemble averaging, and
(iii) select appropriate nonlinear observables to maximize extractable information.
For example,
(i) in physical hardware, external noise at the device-level inherently suppresses spatial synchronization, ensuring the persistence of heterogeneous dynamical responses.
(ii) Parallel device architectures further provide naturally spatially multiplexed channels that can be directly exploited as reservoir nodes (e.g., optical systems~\cite{Optics_spatial_multiplexed} and coupled STOs~\cite{dejong2023virtualreservoiraccelerationcpu}).
(iii) Successful task implementation requires the reservoir state to contain sufficient task-relevant IPC, and as illustrated in Fig.~\ref{fig:equippment_data}b, varying the nonlinear transformation modifies accessible IPC; consequently, such transformations should be selected appropriately for both the physical system and the target tasks (e.g., NARMA10 and CRC).
By following these rules, ERC can be applied to a wide range of dynamical systems as reservoirs, thus widening the set of physical substrates available for computation.
\par
ERC can be regarded as a special class of linear regression. Through weighted ensemble averaging, ERC produces time-invariant representations that correspond to linear combinations of $d \times L$-dimensional reservoir states. This implies that linear regression on these states might yield similar invariance when sufficient trials are available. However, implementing linear regression with these states requires to solve a large size of linear equations. By contrast, ensemble averaging avoids solving such a large size of linear equations and therefore reduces the computational cost, making ensemble average a more practical approach.
\par
While we demonstrate the broad applicability of ERC, several theoretical issues remain. 
The optimal design of nonlinear observables $\phi$ for real physical systems has yet to be established, and the correspondence between the theoretical expansion in Eq.~\eqref{eq:basic_trigometric} and actual dynamics is still unclear. 
Moreover, ERC assumes the ensemble limit $L \to \infty$, leaving open which input-history functions can be extracted with only a finite numbers of ensembles. 
Addressing these challenges will clarify the theoretical foundation of ERC and extend its applicability to a wider range of physical substrates.
\section{Methods}
\subsection{Memory capacity task}
We introduce the concept of memory capacity (MC) to measure a system’s ability to retrieve previous linear inputs.
It is quantified by how accurately the system predicts a past input signal $u_{t-\tau}$ using a time series $X$ that includes the input $u$.
The target is given by $u_{t-\tau}$, and $\hat{y}_{t}$ is the predicted value computed from $X$.
The memory function $M_\tau$ is defined as $M_{\tau}\coloneqq {\rm cor}(u_{t-\tau},\hat{y}_{t})^2$.
Then, MC is computed as $\sum_{\tau=1}^{\infty}M_\tau$.
\par
\subsection{Models}
\subsubsection{Echo state network}
To demonstrate the framework's (i) noise tolerance and (ii) robustness to time- and noise-dependent dynamics, we employ an ESN model with $30$ nodes ($d = 30$), with both input and noise intensities set to $\alpha = \sigma = 0.01$.
As in the case of (i), we set the spectral radius to $\rho = 0.94$.
As in case (ii), we vary it from $0.0$ to $2.5$ in $21$ steps.
Each element of the weight matrix $W$ and the vector $W_{\rm in}$ in derived from a uniform distribution over $[-1, 1]$.
The number of spatially multiplexed systems is set to $L=5 \times 10^5$.
\par
The bifurcation structure of the ESN with respect to the spectral radius $\rho$ is summarized in Fig.~S2, showing transitions between ordered and chaotic regimes.
\subsubsection{Lorenz equation}
For the Lorenz, R\"{o}ssler, and Chua circuits, we use input sequences $u_t$ drawn uniformly from $[-1, 1]$.
In the COPY map, we change the input range to $[0,1]$.
The number of trials $L$ is set to $10^5$.
The nonlinear transformation $\phi$ was selected from the following functions by adopting, in each case, the one that produced the highest capacity:
\begin{equation}
\begin{split}
  &\{x,x^2,\ldots,x^5,\\
  &\cos x,\sin x,\tan x,\cos 2x,\sin 2x,\tan 2x\\
  &e^{x},e^{-x^2},\log(x^2),|x|
  \}.
  \label{eq:function_set}
\end{split}
\end{equation}
Only in the case of the COPY map, $e^{x^2}$ was considered instead of $e^{-x^2}$ in the function set in Eq.~\eqref{eq:function_set}.
\par
The Lorenz model of three states ($x,y,z$) is described by
\begin{equation}
  \begin{split}
    \frac{dx}{dt}&=-\sigma x-\sigma y+\iota u_t\\
    \frac{dy}{dt}&=x(\rho -z)-y,\\
    \frac{dz}{dt}&=xy+\beta z,\\
  \end{split}
\end{equation}
where $\sigma=10,\rho=28,\beta=8/3$, and $\iota=30$.
We solve the equations using the fourth order Runge--Kutta method with a step width of $\Delta t=0.01$.
The maximum conditional Lyapunov exponent is $\lambda_{\rm top}\approx 1.0265$.
\par
In the averaging procedure of ERC, we use the original state $(x,y,z)$, [i.e., $\phi(x)=x$] together with their ensemble-averaged values $\mathbb{E}[x], \mathbb{E}[y]$, and $\mathbb{E}[z]$ as reservoir states.
\subsubsection{R\"{o}ssler equation}
 We describe the R\"{o}ssler system by the following equations:
\begin{equation}
  \begin{split}
    \frac{dx}{dt}&=-y-z,\\
    \frac{dy}{dt}&=x+ay,\\
    \frac{dz}{dt}&=b+xz-cz+\iota u_t,\\
  \end{split}
\end{equation}
where $a=0.2,b=0.2,c=5.7$, and $\iota=0.2$.
We use the fourth-order Runge--Kutta to compute a numerical solution with $\Delta t=0.2$.\par
In the averaging procedure of ERC, we adopt the observation function $\phi(x)=\log(x^2)$.
The maximum conditional Lyapunov exponent is $\lambda_{\rm top}\approx 6\times 10^{-2}$.
\subsubsection{Chua circuit}
The Chua circuit is defined by the following equations:
\begin{equation}
  \begin{split}
    \frac{dx}{dt}&=a(y-x+f(x))+\iota u_t,\\
    \frac{dy}{dt}&=x-y+z,\\
    \frac{dz}{dt}&=-by,\\
    f(x)&=m_1x+\frac{m_0-m_1}{2}(|x+1|-|x-1|),
  \end{split}
\end{equation}
where $a=15.6,b=28.0,m_0=-1.143,m_1=-0.714,\iota=2.0$, and $\Delta t=0.05$. The observation function is $\phi(x)=x$.
The maximum conditional Lyapunov exponent is $\lambda_{\rm top}\approx 0.46$.
\subsubsection{COPY map}
The COPY map setting is defined as follows:
\begin{equation}
  \begin{split}
    &x_{n+1}=2\lambda \tanh(x_{n})\cos\theta_n,\\
    &\theta_{n+1}=[\theta_{n}+ 2\pi\omega ~ (\mbox{mod} 2\pi)]+\iota u_t,
  \end{split}
\end{equation}
where $\lambda=1.5,\omega=\frac{\sqrt{5}-1}{2}$, and $\iota=1$.
In the ERC procedure, we use only $x_n$ as the state variable.
The observation function is $\phi(x)=\exp(x^2)$.
The maximum conditional Lyapunov exponent is $\lambda_{\rm top}\approx 3\times 10^{-4}$.
\subsubsection{NARMA10 task}
The NARMA10 is a standard benchmark to evaluate the prediction accuracy of PRC.
Let \( u_t \) denote the input signal.
 Then, the NARMA10 output $y_t$ evolves according to:
\begin{small}
  \begin{equation}
    \begin{split}
    y_{t+1} & = 0.3 y_t + 0.05 y_t \left( \sum_{i=0}^{9} y_{t-i} \right) + 1.5 \zeta_{t-9} \zeta_{t} + 0.1,\\
    \zeta_t & \coloneqq \delta\left(u_t+\mu \right),
    \end{split}
  \end{equation}
\end{small}where $\delta$ and $\mu$ are parameters that control the range of $\zeta_t$.
\subsection{Observation functions}
\subsubsection{A/D converter}
In Fig.~\ref{fig:equippment_data}c, we numerically use the $n$-bit analog digital (A/D) converter for the observation function. 
First, we normalize the dynamical state $x$ to the $[0,1]$ range, as follows: 
\begin{eqnarray}
y &= \frac{x-x_{\rm min}}{x_{\rm max}-x_{\rm min}}, 
\end{eqnarray}
where $x_{\rm min}$ and $x_{\rm max}$ are the minimum and maximum values of state time series. 
The observation function is described by 
\begin{eqnarray}
  \phi &= \frac{1}{2^n} \lfloor 2^n y \rfloor, 
\end{eqnarray}
where the floor function $\lfloor 2^n y\rfloor$ takes the integer part of $2^ny\in[0, 2^n]$ and then $\phi$ normalizes $\lfloor 2^n y \rfloor$ to the range $[0,1]$ again. 
\subsubsection{Other functions}
In Fig.~\ref{fig:equippment_data}b, we additionally use the observation functions $\phi(x)\in\{x^2, \cos x, {\rm ReLU}(x)\}$. Before applying $\phi$, we normalized the dynamical state $x$ by its time-series standard deviation.
\subsection{IPC and TIPC}
\label{sec:IPC_and_TIPC}
The IPC extends MC to nonlinear processing.
In the framework of IPC, the processed input is represented as a polynomial, and we separately evaluate polynomial terms of different degrees and delays (e.g. $u_{t-1},u_{t-1}u_{t-2}$).
This framework makes it possible to identify the conditions under which all processed inputs can be extracted.
\par
The IPC represents the amount of processed inputs in the instantaneous state.
Using the target output $z_t$ of the delayed nonlinear input, we can calculate its estimate $\hat{z}_t$ using linear regression on the current state.
The IPC is defined as $D \coloneqq 1-{\rm NMSE}(\boldsymbol{z},\boldsymbol{\hat{z}})$.
The TIPC further extends the IPC to handle time-dependent inputs. 
Polynomial terms with temporal structure are included in the TIPC and evaluated on their own.
This allows for the characterization of temporal processing capabilities beyond instantaneous states.
We denote the total IPC or total TIPC by $C_{\rm tot}$.
\par
Let us consider a system state as $\boldsymbol{x}_t\in \mathbb{R}^{d}$ and driven by input $u_t$, and let $\hat{\boldsymbol{x}}_t\in \mathbb{R}^{r}(r\leq d)$ denote its orthonormalized form.
Here, $d$ is the dimension of the original system, and $r$ is the number of independent bases.
Then, $\hat{\boldsymbol{x}}_t$ can be expressed as a function of time $t$, initial value $\boldsymbol{x}_0$, and input history $u_{t-1},u_{t-2},\ldots$ as follows:
\begin{equation}
  \hat{\boldsymbol{x}}_t=F(t,u_{t-1},u_{t-2},\ldots, \boldsymbol{x}_0).
\end{equation}
 The IPC and TIPC are quantified by the norms of the coefficient vectors in the orthogonal polynomial expansion of $F$.
The coefficients associated with bases that depend only on the input history correspond to the IPC, whereas those associated with bases involving temporal variables correspond to the TIPC. 
The sum of the IPC and TIPC is bounded by $r$.
\par
For example, consider the following system:
\begin{equation}
  \begin{split}
    &x\coloneqq u_{t-1}\cos{t}+u^2_{t-2}\sin{2t}+u_{t-1}u_{t-2}+u_{t-1}.
  \end{split}
  \label{eq:TIPC_Example}
\end{equation}
Here, we assume that $t\in[0,2\pi]$ and $u_{t}\in[-1,1]$ for all $t$.
For this system, we expand the target functional in orthogonal polynomials with respect to time and input history and examine the corresponding coefficients.
We use trigonometric functions ($\cos t, \sin 2t$), and Legendre polynomials [$P_i(x)$ for $i = 0, 1, 2$] as orthogonal bases for the temporal and input history, respectively
(see the Supplementary Information for derivation).
\par
This system has $1$st and $2$nd order IPC and TIPC.
We denote the $d$th-order IPC and TIPC as ${\rm IPC}^{(d)}$ and ${\rm TIPC}^{(d)}$, respectively.
They are given by:
\begin{equation}
  \begin{aligned}
    {\rm TIPC}^{(1)} & =\frac{15}{64},& {\rm TIPC}^{(2)} &=\frac{1}{16},\\
    {\rm IPC}^{(1)} &=\frac{15}{32},& {\rm IPC}^{(2)} &=\frac{5}{32}.
    \label{eq:TIPC_result}
  \end{aligned}
\end{equation}
In Fig.~\ref{fig:equippment_data}b, each bar height represents the corresponding value, with the hatched regions denoting the TIPC and the unhatched regions denoting the IPC (see Fig.~S6 for a visual comparison).
\subsection{Spin-torque oscillator}
The experimental setup used in this study was based on a previously reported STO configuration~\cite{Tsunegi_2023} with several parameters adjusted for the present work. The STO was driven by a voltage-modulated binary input sequence. The applied bias voltage was switched between 150 mV and 300 mV, resulting in a controlled shift of the operating point of the oscillator in response to the input data. This voltage modulation induced distinct nonlinear magnetization dynamics, which served as the basis for PRC.
The electrical signal generated by the STO, was amplified by an external circuit and routed through a microfabricated conductor placed above the STO as shown in Fig.~S3.
This conductor produced a high-frequency magnetic field that acted back on the STO, effectively feeding a delayed and amplified version of its own output signal into the device. In this manner, a self-feedback loop was established. The total delay time in the loop was adjusted to 28.6 ns, and the overall loop gain was set to 38 dB when the attenuation value was 0 dB. These parameters were chosen to ensure rich temporal dynamics while maintaining stable oscillation.
The STO voltage output was digitized at a high sampling rate, and the resulting waveform was subsequently processed to construct temporal virtual nodes and ensemble-averaged states, as described in the main text. Except for the modified delay, loop gain, and input-voltage conditions noted above, the measurement procedures followed exactly those of the earlier report.
\subsection{Cyclic redundancy check task}
The CRC is a standard technique to detect an error bit in an input sequence. 
In this procedure, the sender calculates the remainder by dividing the input sequence by a divisor and attaches the remainder to the sequence. 
After receiving the input sequence with the remainder, the receiver recomputes the remainder using the same divisor and checks for consistency.
\par
In this paper, we frame the CRC task as emulating the remainder output.
We use the CRC-4-ITU representation, which uses the divisor of $1101$. 
For a 5-bit input sequence, the target outputs are:
\begin{align}
  y_t = 
  \begin{bmatrix}
    u_t \oplus u_{t-2}\\
    u_t \oplus u_{t-1} \oplus u_{t-3}\\
    u_{t-1} \oplus u_{t-4}
  \end{bmatrix}.
\end{align}
Here, the target $y_t$ is the coefficient vector of the remainder $r(x):=U(x)\bmod g(x)$ (see the Supplementary Information for further derivation), where $U(x$) and $g(x)$ are defined as follows:
\begin{equation*}
  \begin{split}
    &U(x)\coloneqq u_{t}x^4+u_{t-1}x^3+u_{t-2}x^2+u_{t-3}x+u_{t-4},\\
    &g(x)\coloneqq x^3+x+1.
  \end{split}
\end{equation*}
\subsection{Hamming code tasks}
The Hamming code is an important technique for correcting bit-flip errors in a received sequence. 
In this procedure, the sequence is transmitted from a sender to a receiver. 
The sender and receiver perform the encoding and decoding tasks, respectively. 
In the encoding task, using a $4$-bit input sequence $u=[u_t,u_{t-1},u_{t-2},u_{t-3}]^\top$, an additional $3$-bit sequence for correction is computed. 
The $7$-bit sequence, which concatenates them, is given as follows: 
\begin{align*}
  y = G^\top u, 
\end{align*}
where the generator matrix $G$ is described by 
\begin{align*}
  G = 
  \begin{bmatrix}
    1 & 0 & 0 & 0 & 1 & 1 & 0 \\
    0 & 1 & 0 & 0 & 1 & 0 & 1 \\
    0 & 0 & 1 & 0 & 0 & 1 & 1 \\
    0 & 0 & 0 & 1 & 1 & 1 & 1 
  \end{bmatrix}. 
\end{align*}
In the decoding task, the receiver is required to correct the received $7$-bit sequence $v=[v_{t},v_{t-1},v_{t-2},v_{t-3},$
$v_{t-4},v_{t-5},v_{t-6}]^\top$ if the sequence contains a bit-flip error. 
The first four bits of the corrected sequence are described by 
\begin{align*}
  y = \begin{bmatrix}
    (p_1\cdot p_2\cdot \overline{p_3})\oplus \overline{v_{t}} \\
    (\overline{p_1}\cdot p_2\cdot p_3)\oplus \overline{v_{t-1}} \\
    (p_1\cdot \overline{p_2}\cdot p_3)\oplus \overline{v_{t-2}} \\
    (p_1\cdot p_2\cdot p_3)\oplus \overline{v_{t-3}} 
  \end{bmatrix}, 
\end{align*}
where the parity bits $[p_1,p_2,p_3]$ are described by 
\begin{align*}
  \begin{bmatrix}
    p_1 \\
    p_2 \\
    p_3
  \end{bmatrix} 
  = Hv = 
  \begin{bmatrix}
    v_{t}\oplus v_{t-1}\oplus v_{t-3}\oplus v_{t-4} \\
    v_{t}\oplus v_{t-2}\oplus v_{t-3}\oplus v_{t-5} \\
    v_{t-1}\oplus v_{t-2}\oplus v_{t-3}\oplus v_{t-6}
  \end{bmatrix},
\end{align*}
with the parity-check matrix 
\begin{align*}
  H = 
  \begin{bmatrix}
    1 & 1 & 0 & 1 & 1 & 0 & 0 \\
    1 & 0 & 1 & 1 & 0 & 1 & 0 \\
    0 & 1 & 1 & 1 & 0 & 0 & 1
  \end{bmatrix}.
\end{align*}
\subsection{Echo state property}
To evaluate and ensure the success of RC, we introduce the echo state property (ESP), a standard indicator of its computational performance. Since RC predicts outputs through a linear prediction by a linear readout of system states, these states must be reproducible functions of the input history.
The ESP has two definitions: First, after sufficiently long input injection, the state should be expressible as a function of input history $U=\{u_{t-\tau}\}_{\tau=0}^{\infty}$ as follows:
\begin{equation*}
  \boldsymbol{x}=F(u_t,u_{t-\tau_1},u_{t-\tau_2},\ldots).
  \label{eq:echo_func}
\end{equation*}
The function $F$ is called an echo function, which guarantees a reproducible response to an identical input sequence.
The state of an input-driven dynamical system depends on the input history, time, and the initial value.
The first definition of the ESP requires that the dependence on both time and the initial value vanish. The second definition uses common-signal-induced synchronization, in which two systems with different initial values receive the same input and their states eventually converge; this convergence indicates that, after long input injection, the states become independent of the initial values that when injecting input for a long time, the states are independent of initial values.
However, these definitions are not fully equivalent in terms of variable dependence.
Synchronization removes dependence on the initial value and but does not guarantee that the state depends solely on the input history [e.g., $x(t,U,x_0)\coloneqq u_{t}\cos t+e^{-t}x_0$].
\section{Acknowledgements}
This work was supported by JST BOOST (Grant No. JPMJBS2405) and JST CREST (Grant No. JPMJCR2014).
We thank Dr. Hitoshi Kubota (AIST) for assisting with device 
microfabrication and for providing the SEM images used in this study.

\clearpage
\onecolumngrid
\section*{Supplementary Information:\\
Ensemble Reservoir Computing for Physical Systems}
\twocolumngrid

\section{Time-invariant QUANTITIES FROM ensemble averaging}
Before stating the theorem, we introduce the notation for the weights $\{\omega_l^{(L)}\}_{l=1}^L$ and the definition of an analytic function, as follows:
\begin{align*}
    \mathcal{W}_L & \coloneqq \{\omega_l^{(L)}\}_{l=1}^L, \\
    \mathcal{W}_L^{+} & \coloneqq\{\omega_l^{(L)} \in \mathcal{W}_L \mid \omega_l^{(L)}>0,\ l=1,2, \ldots,L\},\\ 
    \mathcal{W}_L^{-} & \coloneqq\{\omega_l^{(L)} \in \mathcal{W}_L \mid \omega_l^{(L)}<0,\ l=1,2, \ldots,L\}.
\end{align*}
A function $\phi\colon (a, b)\to \R~(a, b\in\R, a<b)$ is said to be analytic on $(a,b)$ if, for any point $c \in (a,b)$, there exists a radius $R > 0$ such that
\begin{align}
    \phi(x) = \sum_{n=0}^{\infty} \frac{\phi^{(n)}(c)}{n!}\,(x - c)^n,\quad
    |x - c| < R,
    \label{eq:phi_expansion}
\end{align}
where the series converges and its sum equals $\phi(x)$. Here, $\phi^{(n)}(c)$ is the $n$-th derivative of $\phi$ at point $c$.
\begin{supptheorem}[Ensemble average transformation]\label{thm:amsemble}
Suppose that the weights $\mathcal{W}_L = \{\omega_l^{(L)}\}_{l=1}^L$ satisfy
\[
    \max_{\omega_l^{(L)}\in \mathcal{W}_L} |\omega_l^{(L)}| \rightarrow 0 \quad (L\to\infty).
\]
We consider the following two cases for the error factor $E$:
\begin{itemize}
\item[(i)]
$E$ depends only on noise  $(\mbox{i.e., } E=E^{(l)}=\{v^{(l)}\}):$
\par
\quad
The system state is described by $x=x(U,v)$.
Here, $v$
is a random variable distributed on $(a,b)$ and i.i.d.\ across spatially multiplexed systems.
For the weights $\mathcal{W}_L = \{\omega_l^{(L)}\}_{l=1}^L$, there exist non-negative constants $C_1$ and $C_2$ with $\max\{C_1, C_2\} > 0$ and independent of $L$, such that
\[
    \qquad \sum_{\omega_l^{(L)}\in \mathcal{W}_L^{+}} \omega_l^{(L)} = C_1, \quad
    \sum_{\omega_l^{(L)}\in \mathcal{W}_L^{-}} |\omega_l^{(L)}| = C_2.
\]
\item[(ii)]
$E$ depends only on a temporal fluctuation ${\rm (} \mbox{i.e., } E = E^{(l)} = \{t,\boldsymbol{x}_0^{(l)}\}{\rm ):}$
\par
\quad
Let $x(t) = x(t, U,\boldsymbol{x}_0)$, $t \geq 0$, be a function of $U$ and $\boldsymbol{x}_0$ satisfying $x(t)\in (a,b)$ for $a, b\in\R~(a<b)$.
Let a sequence of initial values $\{\boldsymbol{x}^{(l)}_0\}_{l=1}^\infty$ be given.
Suppose that, for each $l\in\N$, the relation:
\begin{align}
\qquad & x(t, U,\boldsymbol{x}^{(l)}_0) \notag \\
 & \simeq \sum_{j=1}^{N} \Bigl\{
a_j\cos \bigl(f_j+\theta^{(l)}_j \bigr)
+ b_j\sin \bigl(f_j+\theta^{(l)}_j \bigr)\Bigr\}+ C
\label{approx:x_t_U_x0}
\end{align}
holds with smooth functions $a_j = a_j(U)$, $b_j = b_j(U)$, $f_j = f_j(t,U)$, and $C = C(U)$ and with phases $\theta_j^{(l)}$, $j=1,\ldots,N$.
Suppose that for each $l$, the dependence on the initial value is indirectly expressed by i.i.d.~random phases $\{\theta_j^{(l)}\}_{j=1}^N$ that follow a uniform distribution in the range of $(-\pi,\pi)$ and that the function~$\phi\colon (a,b) \rightarrow \R$ is analytic.
Suppose further that the weights $\mathcal{W}_L = \{\omega^{(L)}_l\}_{l=1}^{L}$ satisfy the following:
\begin{itemize}
    \item[{\rm (}ii-1{\rm )}]
    All weights in $\mathcal{W}_L$ have the same sign (i.e., $\mathcal{W}_L = \mathcal{W}_L^{+}$ or $\mathcal{W}_L = \mathcal{W}_L^{-}$).
    \item[{\rm (}ii-2{\rm )}]
    There exists a positive constant $C_3$, independent of $L$, such that
    \begin{align*}
        \sum_{l=1}^{L} |\omega_l^{(L)}| = C_3.
    \end{align*}
\end{itemize}
\end{itemize}
Suppose that (i) or (ii) holds and that the weighted average:
\[
\sum_{l=1}^{L}\omega_l^{(L)}\phi(x(U, E^{(l)}))
\]
converges as $L \to \infty$; that is, $\mathbb{E}[\phi(x)]$ in Eq.~{\rm (1)} is well-defined.
Then, $\mathbb{E}[\phi(x)]$ is time-invariant.
\end{supptheorem}
\begin{proof}
To begin with, we prove the theorem for case~(i).
Decomposing the finite weighted average into two terms and letting $L \to \infty$, we obtain
\begin{align*}
    &\sum_{l=1}^L \omega_l^{(L)} \phi \left( x(U, v^{(l)}) \right) \\
    &= \sum_{\omega_l^{(L)}\in\mathcal{W}_L^{+}}\omega_l^{(L)} \phi \bigl( x(U, v^{(l)}) \bigr) +
    \sum_{\omega_l^{(L)}\in\mathcal{W}_L^{-}}\omega_l^{(L)} \phi \bigl( x(U, v^{(l)}) \bigr) \\
    & \to \int_{S_1}\phi( x(U, v))P(v)dv + \int_{S_2}\phi( x(U, v))P(v)dv,
\end{align*}
where $S_1\cup S_2=(a,b)$, $S_1\cap S_2=\varnothing$, $P:(a,b)\rightarrow\R$ is a distribution of $v$, and the Lebesgue measures of $S_1$ and $S_2$ are $C_1$ and $C_2$, respectively.
Since the two integrations over $S_1$ and $S_2$ above are independent of noise~$v$, $\mathbb{E}[\phi(x)]$ is time-invariant.
\par
Next, we prove the theorem for case (ii).
As $\phi$ is analytic, we have the expression of $\phi$ in Eq.~\eqref{eq:phi_expansion}.
Therefore, it is sufficient to consider the $n$-th power case. 
By substituting the relation in Eq.~\eqref{approx:x_t_U_x0} into $x$, the $n$-th power of $x$ is transformed into:
\begin{equation}
    x^n\simeq\sum_{k=0}^{n}\binom{n}{k}\left( \sum_{j=1}^{N}r_j\cos(\theta^{(l)}_j+\eta_j)\right)^{n-k}C^{k},
    \label{eq:composition_form}
\end{equation}
where $r_j \coloneqq \sqrt{a_j^2+b_j^2}$ is time-invariant and $\eta_j\coloneqq \arctan(\frac{a_j\cos f_j+b_j\sin f_j}{-a_j\sin f_j+b_j\cos f_j})$ is time-variant. When we choose the weights $\omega_l^{(L)} = 1/L$~$(l=1,\ldots,L)$, which satisfies the assumption in (ii-1) and (ii-2). Then, the time-invariance of the ensemble average $\mathbb{E}[x^n]$ is shown as follows:
\begin{align}
    &\mathbb{E}[x^n]
    \simeq \mathbb{E}\left[\sum_{k=0}^{n}\binom{n}{k}C^{k}\left( \sum_{j=1}^{N}r_j\cos(\theta^{(l)}_j+\eta_j)\right)^{n-k}\right]\tag{S.~3a}\label{eq:one}\\
    &=\sum_{k=0}^{n}\frac{n!}{k!}C^{k}\mathbb{E}\left[ \sum_{\substack{l_1,\ldots,l_N\geq 0\\
    l_1+\ldots+l_N=n-k}}\prod_{j=1}^{N}\frac{r^{l_j}_j}{l_j!}\cos(\theta^{(l)}_j+\eta_j)^{l_j}\right]\tag{S.~3b}\label{eq:two}\\
    &=\sum_{k=0}^{n}\frac{n!}{k!}C^{k}\sum_{\substack{l_1,\ldots,l_N\geq 0\\
    l_1+\ldots+l_N=n-k}}\prod_{j=1}^{N}\frac{r_j^{l_j}}{l_j!}\int_{-\pi}^{\pi}\cos(\theta_j+\eta_j)^{l_j}d\theta_j\tag{S.~3c}\label{eq:three}\\
    &=\sum_{k=0}^{n}\frac{n!}{k!}C^{k} \sum_{\substack{l_1,\ldots,l_N\geq 0\\
    l_1+\ldots+l_N=n-k}}\prod_{j=1}^{N}\frac{r_j^{l_j}}{l_j!}\int_{-\pi}^{\pi}\cos^{l_j}\theta_jd\theta_j\tag{S.~3d}\label{eq:four}\\
    &=\sum_{k=0}^{n}\frac{n!}{k!}C^{k} \sum_{\substack{l_1,\ldots,l_N\geq 0\\
    l_1+\ldots+l_N=n-k\\ l_1,\ldots, l_N \colon \mbox{even}
    }}\prod_{j=1}^{N}\frac{r_j^{l_j}}{l_j!}\frac{2\pi}{2^{l_j}}\binom{l_j}{l_j/2}\tag{S.~3e}\label{eq:ri_Prod}
\end{align}
where the right-hand side of Eq.~\eqref{eq:one} is obtained by substituting Eq.~\eqref{eq:composition_form} into $x^n$, Eq.~\eqref{eq:two} follows from the multinomial theorem: for any $m\in \mathbb{N}\cup\{0\}$,
\begin{equation}
    \left(\sum_{j=1}^{N}\alpha_j\right)^{m}=\sum_{\substack{l_1,\ldots,l_N\geq 0\\
    l_1+\ldots+l_N=n-k}}\frac{m!}{l_1!\ldots l_N!}\prod_{j=1}^{N}\alpha_j^{l_j}.
\end{equation}
Moreover, we use the identity
\[
\binom{n}{k}\frac{(n-k)!}{l_1!\cdots l_N!}
=\frac{n!}{k!}\frac{1}{l_1!\cdots l_N!}=\frac{n!}{k!}\prod_{j=1}^{N}\frac{1}{l_j!},
\]
with $\alpha_j=\cos(\theta_j^{(l)}+\eta_j)$ and $m=n-k$, Eq.~\eqref{eq:three} follows from the independence of $\theta_i$ and $\theta_j$ for $i\neq j$, in Eq.~\eqref{eq:four} is a consequence from the 2$\pi$-periodicity of the cosine function, and Eq.~\eqref{eq:ri_Prod} follows from the identity:
\begin{equation}
    \int_{-\pi}^{\pi}\cos^{l_j} \theta_j d\theta_j=
    \displaystyle
    \left\{
    \begin{array}{cc}
    \displaystyle 0 & (l_j \mbox{ is odd})\\
    \displaystyle \frac{2\pi}{2^{l_j}}\binom{l_j}{l_j/2} & (l_j\mbox{ is even})
    \end{array}
    \right..
    \label{eq:cos_prod_integral}
\end{equation}
For any weights with (ii-1) and (ii-2), the weighted sum converges to the same integral by uniqueness.
Hence, the ensemble average of $x^{n}$ is time-invariant.
\par
Second, we consider the ensemble average of $\phi(x)$. Letting $\phi_{N}$ be a partial sum defined by $\phi_N(x)\coloneqq \sum_{n=0}^{N}\frac{\phi^{(n)}(c)}{n!}(x-c)^n$, we have
\begin{equation}
    |\phi_N(x)|\leq\sum_{n=0}^{\infty}\left| \frac{\phi^{(n)}(c)}{n!}\right | |(x-c)^n|,
    \label{eq:phi_N}
\end{equation}
\begin{figure*}[t]
    \centering
    \includegraphics[width=\textwidth]{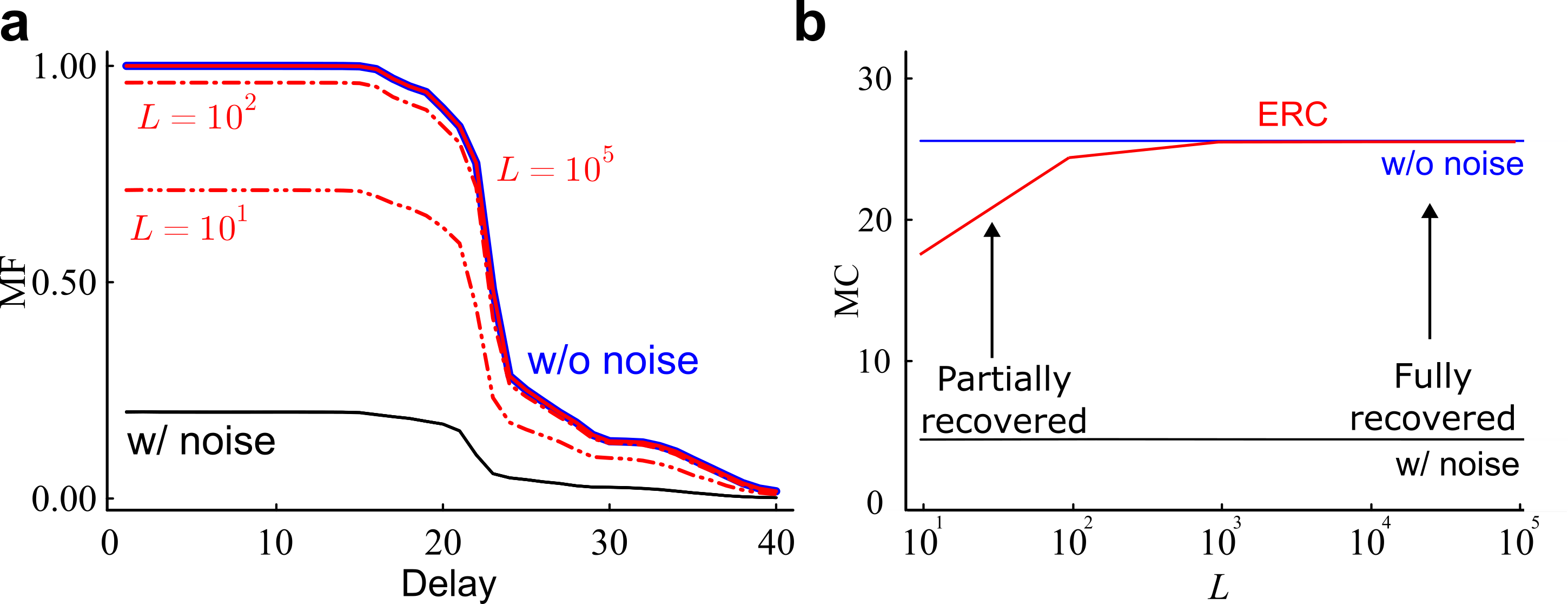}
    \caption{
    \textbf{Finite-size comparison of the total retrievable MC.}
    \textbf{a}) Some memory functions are illustrated: a standard ESN (blue) and an ESN with noise (black).
    The red line shows ERC results with $L=10$, $L=10^2$ (dash), and $L=10^5$ (solid).
    \textbf{b}) The MC as a function of the number of spatially multiplexed systems.
    The blue and black solid lines show the MC without noise and with noise, respectively.
    The red line shows the MC for spatially multiplexed systems with $L\in \{10,10^2,10^3,10^4,10^5\}$.
    }
    \label{eq:finite_size_ESN}
\end{figure*}where the right-hand side of Eq.~\eqref{eq:phi_N} converges as $\phi$ is analytic. By the Lebesgue dominated convergence theorem, we obtain:
\begin{equation}
    \mathbb{E}[\phi(x)]=\sum_{n=0}^{\infty}\frac{\phi^{(n)}(c)}{n!}\mathbb{E}[(x-c)^n].
\end{equation}
From the binomial expansion, 
$\mathbb{E}[(x-c)^{n}]$ is time-invariant for any $n \in \mathbb{N}$, which completes the proof.
\end{proof}
\begin{suppcor}\label{rem:sin_phi}
An important consequence of Theorem~\ref{thm:amsemble} with the assumption (ii) is stated below.
Let $\varphi$ be any function in $L^2(a,b)$.
Then $\varphi$ can be expanded as follows:
\begin{equation}
        \frac{a_0}{2}+\sum_{n=1}^{\infty} \left\{ a_n\cos\bigl( 2n\pi \, y(x)\bigr) + b_n\sin\bigl( 2n\pi \, y(x)\bigr) \right\},
    \label{eq:fourier_expansion}
\end{equation}
where $a_n$ and $b_n$ are given by, for $n\in \N\cup\{0\}$,
    \begin{align*}
        a_n & \coloneqq  \frac{2}{b - a} \int_a^b \varphi(x) \cos\left( 2n\pi \frac{x-a}{b-a} \right) dx,\\
        b_n & \coloneqq  \frac{2}{b - a} \int_a^b \varphi(x) \sin\left( 2n\pi \frac{x-a}{b-a}\right) dx,
    \end{align*}
and $y(x)\coloneqq (x-a)/(b-a)$.
Suppose that $\{a_n\}$ and $\{b_n\}$ belong to $\ell^{1}$.
Then, $\mathbb{E}[\varphi(x)]$ is also time-invariant.
\end{suppcor}
\begin{proof}
It is sufficient to show that the series converges pointwise almost everywhere and is dominated by a function $g\in L^1(a,b)$.
Let $\varphi_N$ denote the partial sum truncated after $N$ terms in Eq.~\eqref{eq:fourier_expansion}.
Then, the absolute value of $\varphi_N$ satisfies the following inequality:
\begin{equation}
    |\varphi_N(x)|\leq \frac{a_0}{2}+\sum_{n=0}^{\infty}\bigl( |a_n|+|b_n| \bigr).
\end{equation}
Since from $a_n, b_n\in \ell^1$, the right-hand side is bounded, and we can define $g$ accordingly.
\end{proof}
\subsection{Performance of ERC for finite-size systems}
The ERC method relies on an ensemble average transformation that extracts a time-invariant value in the limit $L\rightarrow\infty$ given in Eq.~{\rm (1)}.
In practice, this limit is approximated using a finite number $L$ of spatially multiplexed systems.
\par
We evaluate the finite-size effect numerically using the ESN defined in Eq.~{\rm (2)}.
All settings, except $L$, are described in Methods.
Since noise and input are reflected in the same vector $W_{\rm in}$, the input information extractable by ERC is bounded above by that of a standard ESN.
Figures~\ref{eq:finite_size_ESN}a and b compare MFs and MC, respectively.
The MC lost by noise injection is fully recovered for $L>10^3$.
However, it is only partially recovered for $L=10$ and $L=100$, which represent realistic values in practical settings.
\vspace{1.0\baselineskip}
\section{Time-invariant definition}
\begin{suppdfn}\label{def:invariant}
Consider $x = x(U,t,x_0,v_t)$.
For any fixed $h>0$, we define $x_1$ which is shifted the time by keeping the input history fixed, as follows: 
\begin{equation}
    x_1\coloneqq x(U, t+h, \boldsymbol{x}_0, v_{t+h}).
\end{equation}
If $x=x_1$ for any $t>0$, then $x$ is time-invariant.
\end{suppdfn}
\section{Bifurcation of ESNs}
The dynamical behavior of an ESN is governed by the spectral radius $\rho$ of the connection matrix $W$. 
The ESN exhibits fixed-point behavior when $\rho < 1$, while periodic or chaotic behavior appears when $\rho > 1$~\cite{ESN_jeager}.
Similar behavior was observed in our system (Fig.~\ref{fig:ESN_bifurcation}), with chaotic dynamics emerging for $\rho > 2$. Chaos was quantified by the Lyapunov exponent defined below.
\begin{figure*}[t]
    \centering
    \includegraphics[width=\textwidth]{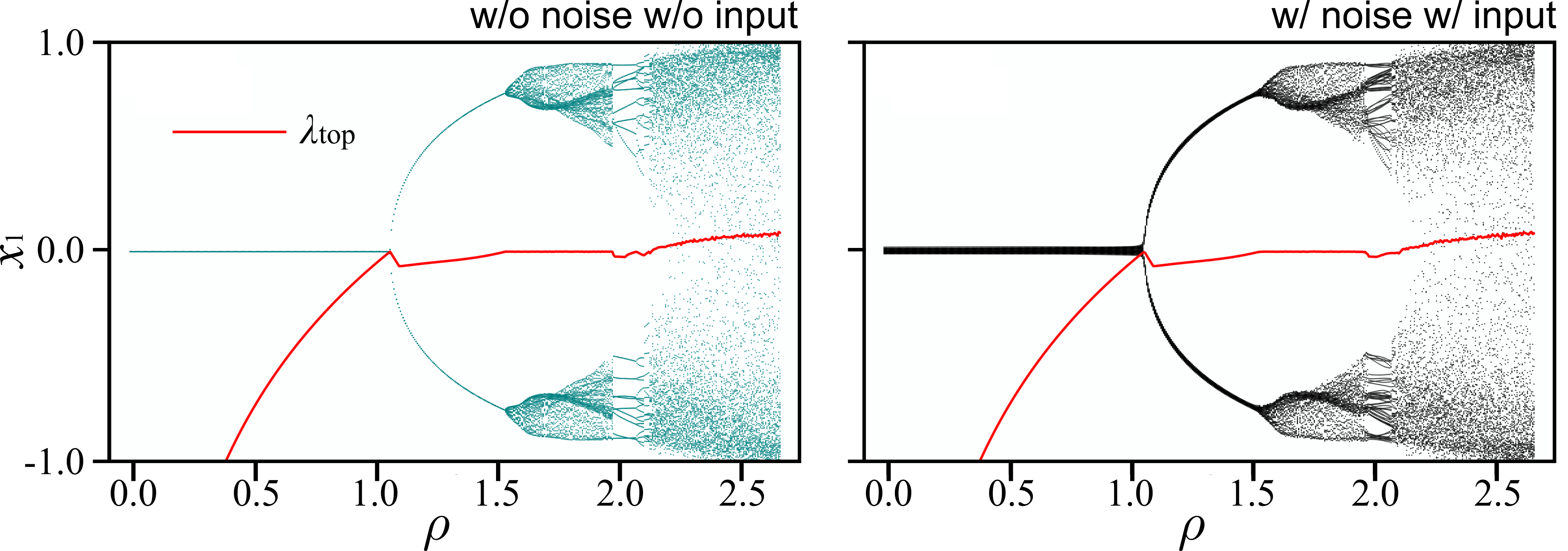}
    \caption{\textbf{Bifurcation diagram of an ESN.} Bifurcation diagrams of the first component of the ESN state and the maximum Lyapunov exponent $\lambda_{\mathrm{top}}$ (red) for two cases: with noise/input (left) and without noise/input (right).
    $\lambda_{\rm top}$ serves as the maximum Lyapunov exponent (no input) or the maximum conditional Lyapunov exponent (with input/noise).
    The horizontal axes represent the spectral radius $\rho$ of the ESN.}
    \label{fig:ESN_bifurcation}
\end{figure*}
\subsection{Maximum Lyapunov exponent}
The maximum Lyapunov exponent $\lambda_{\rm top}$ is defined as follows:
\begin{equation}
    \lambda_{\rm top}\coloneqq \lim_{t\rightarrow \infty}\frac{1}{t}\log\frac{\|\delta(t)\|_\ell^2}{\|\delta(0)\|_\ell^2},
\end{equation}where $\delta(t)$ denotes the distance between two trajectories. 
These trajectories have different initial values in the same $d$-dimensional system.
$\lambda_{\rm top}$ is a criterion for determining whether the system is chaotic.
More precisely, we call $\lambda_{\rm top}$ the maximum conditional Lyapunov exponent when the system is driven by an external input. We calculated $\lambda_{\rm top}$ using an algorithm described by J. C. Sprott~\cite{Lyapunov_exponent_numerically}.
Specifically, we track the divergence of nearby trajectories by repeatedly renormalizing their separation vector and then averaged the logarithmic growth rate over long trajectories after eliminating transients.
\section{STO setup details}
\begin{figure}[b]
    \centering
    \includegraphics[width=0.48\textwidth]{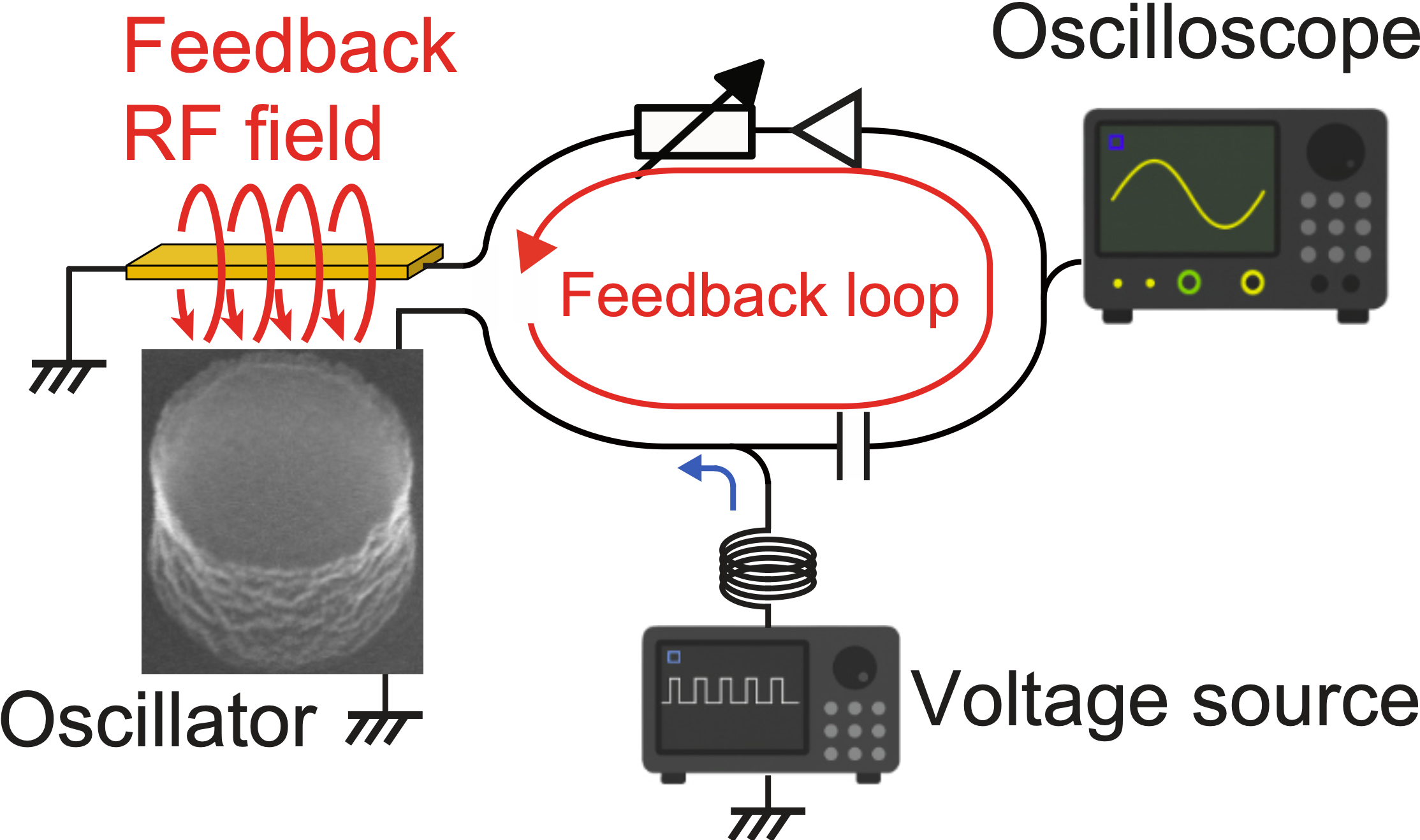}
    \caption{\textbf{Schematic of the experimental setup used for STO reservoir computing.} The STO image shown corresponds to the resist pillar formed during the lithography step and is provided for illustrative purposes only; it does not represent the final metallic multilayer structure of the device.}
    \label{fig:STO_expeimental_setting}
\end{figure}
The STO used in this study is a device that converts the nonlinear magnetization dynamics of a nanoscale ferromagnet into an electrical signal. We control the nonlinear dynamics of the STO through a self-feedback loop and perform PRC by exploiting these dynamics. The device structure and feedback circuit are nearly identical to those reported previously~\cite{Tsunegi_2023} (see Fig.~\ref{fig:STO_expeimental_setting}).
Moreover, we list the MC and IPC reported in representative prior studies on STO-based RC (see Table~\ref{tab:comparison_spin}).
\section{Derivation of CRC}
We provide an explicit derivation of Eq.~(17) using the polynomials \(U(x)\) and \(g(x)\) defined in Methods.
The XOR-based polynomial long-division steps for computing the remainder \(r(x)= U(x)\bmod g(x)\) are summarized in Table~\ref{tab:crc_steps}, where \(\oplus\) denotes the XOR operation.
The final row of Table~\ref{tab:crc_steps} gives the three coefficients of \(r(x)\), and we define \(y_t\) as the corresponding 3-bit coefficient vector.
\begin{table}[H]
  \caption{XOR-based long-division steps for computing the CRC remainder.}
  \label{tab:crc_steps}
  \centering
  {\footnotesize
  \setlength{\tabcolsep}{2.5pt}
  \renewcommand{\arraystretch}{1.05} 
  \everymath{\textstyle}
  \begin{tabular}{@{}ccccc@{}}
    $u_t$ & $u_{t-1}$ & $u_{t-2}$ & $u_{t-3}$ & $u_{t-4}$ \\
    $u_t\!\cdot\!1$ & $u_t\!\cdot\!0$ & $u_t\!\cdot\!1$ & $u_t\!\cdot\!1$ & \\
    \cline{1-4}
    $u_t\oplus u_t$ & $u_{t-1}$ & $u_t\oplus u_{t-2}$ & $u_t\oplus u_{t-3}$ & $u_{t-4}$\\
    $(=0)$ & & & & \\
    & $u_{t-1}\!\cdot\!1$ & $u_{t-1}\!\cdot\!0$ & $u_{t-1}\!\cdot\!1$ & $u_{t-1}\!\cdot\!1$ \\
    \cline{2-5}
    & $u_{t-1}\oplus u_{t-1}$ & $u_t\oplus u_{t-2}$ & $u_t\oplus u_{t-1}\oplus u_{t-3}$ & $u_{t-1}\oplus u_{t-4}$\\
    &$(=0)$& & &\\
  \end{tabular}
  }
\end{table}
\begin{table*}[t]
    \centering
    \caption{IPCs of STO-based reservoirs. Note that STO stands for spin-torque oscillator.}
    \resizebox{\textwidth}{!}{
    \begin{tabular}{c|c|c|c|c|c|c|c|c|c}
        \label{tab:comparison_spin}
         & \multicolumn{2}{c|}{Input} & Input& RC & \# of total & Virtual & MC & Total & Reference \\ 
         & \multicolumn{2}{c|}{type} & modulation & framework & states & node & & IPC & \\ \hline \hline
         Experiment &  Binary & Voltage & Amplitude & Ensemble RC & 13,200 & Yes & 7.8 & 116 & This work\\
         & Binary & Voltage & Amplitude & Generalized RC & 100 & Yes & 4.3 & N/A & Kubota et al., 2025~\cite{GRC}\\
         & Real-valued & Voltage & Amplitude & RC & 200 & Yes & 2 & 5.6 & Tsunegi et al., 2023~\cite{Tsunegi_2023}\\
         & Binary & RF magnetic field & Phase & RC & 250 & Yes & 3.6 & N/A & Tsunegi et al., 2019~\cite{Tsunegi_2019}\\
         & Binary & Voltage & Amplitude & RC & 200 & Yes & 1.8 & N/A & Tsunegi et al., 2018~\cite{Tsunegi2018MemoryCapacitySTO}\\ \hline
         Simulation & Binary & Current & Amplitude & RC & 250 & Yes & 4 & N/A & Yamaguchi et al., 2023~\cite{Spin_Reservoir}\\
         & Real-valued & Magnetic field & Amplitude & RC & 100 & Yes & 21 & N/A & Akashi et al., 2022~\cite{Akashi_2022}\\
         & & (quasi-static) & & & & & & & \\
         & Real-valued & Current & Amplitude & RC & 3 & No & 0.8 & 2 & Akashi et al., 2020~\cite{Akashi_2020}
     \end{tabular}
     }
\end{table*}
\section{LIMITATIONS OF ERC}\label{sec:ERC_Limitation}
\subsection{Conditions for ERC applicability}
In the temporal fluctuations case, the ensemble average transformation requires phase diversity among trials. This requirement motivates the assumption that the state $x$ admits an expansion with independent random phases $\{\theta_j^{(l)}\}_{j=1}^{N_{\ast}}$: averaging over the dispersed phases eliminates the explicit time dependence in Eq.~{\rm (3)} (see Theorem S.1).
By contrast, when the states become spatially synchronized, their dependence on initial conditions disappears, leading to a collapse of phase diversity.
Random input induces synchronization of limit-cycle trajectories, eliminating this diversity (see Fig.~\ref{fig:comparing_noise_affection}a, upper and center).
To restore phase diversity, we introduce noise
(Fig.~\ref{fig:comparing_noise_affection}a, lower), rewriting Eq.~{\rm (3)} by replacing the random phase $\theta_j^{(l)}$ with a noise-induced phases $\{\psi_{t,j}^{(l)}\}_{j=1}^{N_{\ast}}$ as follows:
\begin{equation}
 \displaystyle
x(t,U,\boldsymbol{x}^{(l)}_0)\simeq\sum_{j=1}^{N_{*}}a_j\cos(f_j+\psi^{(l)}_{t,j})+b_j\sin(f_j+\psi^{(l)}_{t,j})+C.
    \label{eq:basic_trigometric_psi}
\end{equation}
Furthermore, if the noise dependence enters only through \(f_j\), the phase \(\psi_{t,j}^{(l)}\) can be written as
\begin{equation}
    \psi_{t,j}^{(l)} = \arctan\!\left(\frac{Y_j(V^{(l)})}{X_j(V^{(l)})}\right),
    \label{eq:noise_induced_phase}
\end{equation}
where \(X_j(V^{(l)})\) and \(Y_j(V^{(l)})\) are defined in Eq.~\eqref{eq:phase_XY}.
\begin{equation}
    \begin{split}
        &X_j(V^{(l)}) \coloneqq a_j\cos(f_j(V^{(l)})) + b_j\sin(f_j(V^{(l)})),\\
        &Y_j(V^{(l)}) \coloneqq -a_j\sin(f_j(V^{(l)})) + b_j\cos(f_j(V^{(l)})).
    \end{split}
    \label{eq:phase_XY}
\end{equation}
Since the input history is shared across all spatially multiplexed systems, all trial-to-trial variability is solely attributed to noise history $V^{(l)}:=\{v^{(l)}_t\}_{t\geq 0}$.
In particular, if $Y_j(V^{(l)})/X_j(V^{(l)})$ follows Cauchy distributions, $\psi^{(l)}_{t,j}$ is uniformly distributed (see Theorem~\ref{thm:normal_dis}). Such noise-induced phases restore phase diversity that would otherwise be lost due to spatial synchronization. As a result, ensemble averaging can extract time-invariant quantities even in the presence of temporal fluctuations.
\par
\begin{figure}[t]
    \includegraphics[width=0.4\textwidth]{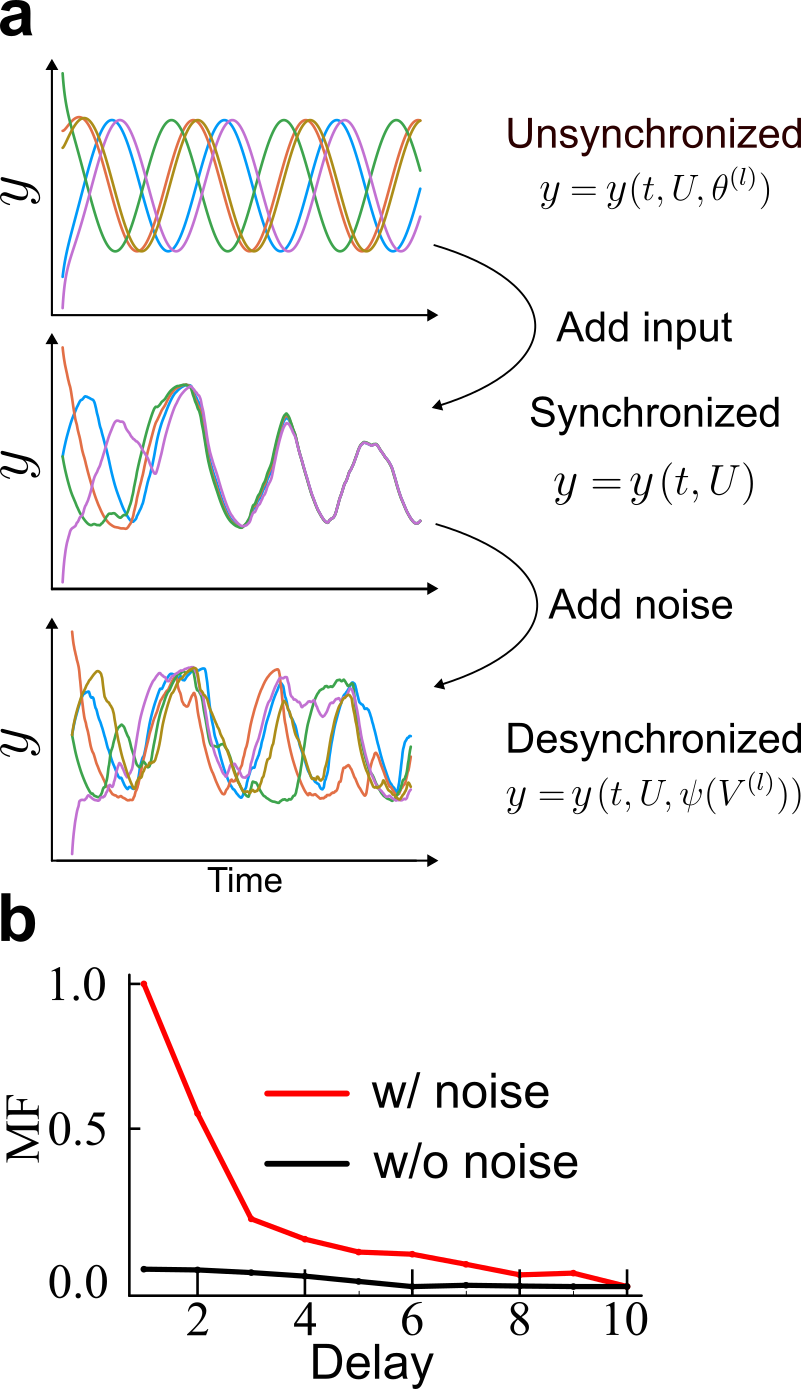}
    \caption{\textbf{Noise-induced desynchronization improves memory performance}.
    \textbf{a}) Synchronization comparison for the $y$-component of the Stuart–Landau system in Eq.~\eqref{eq:spin_SL_oscillator} under three conditions: without input (top), with input only (middle), and with both input and noise (bottom).
    \textbf{b}) The MF for the synchronized (black, noise-free) and desynchronized (red, noisy) cases. 
    The MF was evaluated using an ensemble of polynomial functions \(x, x^2, \dots, x^5\), with the number of spatially multiplexed systems $L=10^6$.
    }
    \label{fig:comparing_noise_affection}
\end{figure}
To numerically demonstrate how noise $v_t^{(l)}$ prevents input-induced synchronization and recovers computational capabilities, we use the Stuart--Landau system, which receives white noise as a discrete-time input $u_t$ in the $x$-direction as follows:
\begin{equation}
    \begin{split}
        &\dot{x} = \alpha x - \beta y - x(x^2 + y^2) + \sigma_1 u_t + \sigma_2 v^{(l)}_t, \\
        &\dot{y} = \beta x + \alpha y - y(x^2 + y^2),
    \end{split}
    \label{eq:spin_SL_oscillator}
\end{equation}
where $\alpha=1$, $\beta=1$, $\sigma_1=0.8$, and $\sigma_1=0.1$.
The number of spatially multiplexed systems is set to $L=10^6$ and the discrete time step is $\Delta t=0.01$. 
$v_t$ follows the normal distribution.
In the absence of noise, the random input $u_t$ induces synchronization among the oscillators.
Figure~\ref{fig:comparing_noise_affection}a shows the time series of the Stuart--Landau system under three conditions: neither input nor noise (upper), input without noise (center), and both input and noise (lower).
Without either noise or input, the oscillators remain unsynchronized.
When input is applied, the common forcing induces synchronization (center).
Finally, injecting different noise realizations into the spatially multiplexed systems breaks the spatially synchronization.
Figure~\ref{fig:comparing_noise_affection}b shows the MF of the Stuart--Landau oscillators without noise (black) and with noise (red); counterintuitively, the MF with noise outperforms that without noise.
Note that the MFs were calculated using ensemble averages of the first to fifth powers of $x$ ($\mathbb{E}[x], \ldots, \mathbb{E}[x^5]$).
The large MF observed with noise suggests that ERC is effective only for spatially multiplexed systems without synchronization, which is a limitation of the ERC framework.
However, even in the absence of synchronization, there remains a fundamental limit to the amount of input information that can be extracted.
\subsubsection{Disturbance collapses the synchronization}
Let us add random perturbations to a system that exhibits synchronization.
\begin{supptheorem}[Desynchronization by noise]\label{thm:normal_dis}
We assume that noise enters only through $f_j$ in Eq.~{\rm (3)}, and variables $Y_j(V^{(l)})/X_j(V^{(l)})$ follows Cauchy distribution.
Then, the noise-induced phase $\psi_{t,j}^{(l)}$ follows a uniform distribution.
Here, $X_j$ and $Y_j$ are defined in Eq.~\eqref{eq:phase_XY}.
\end{supptheorem}
\begin{proof}
    From the assumption, $Y_j(V^{(l)})/X_j(V^{(l)})$ follows a Cauchy distribution.
    It is a standard fact that if $R_j$ follows a uniform distribution, then $\tan R_j$ follows the Cauchy distribution. Therefore, the noise-induced phase $\psi_{t,j}^{(l)} = \arctan(R_j^{(l)})$ follows a uniform distribution.
\end{proof}
\subsection{Limits on extractable information}
We characterize the extractable time-invariant quantities using the IPC.
In this study, we derive two limitations on the input information that can be extracted by the ERC.
First, for any $\phi$ satisfying the assumptions of Theorem~\ref{thm:amsemble}, $\mathbb{E}[\phi(x)]$ is expressed as a function of $r_k\coloneqq \sqrt{a_k^2+b_k^2}$ (for $k=1,\ldots,N_{\ast}$) and $C$, where $N_{\ast}$ is equal $N$ in Eq.~\eqref{approx:x_t_U_x0}.
In addition, to determine $r^2_1,\ldots,r^2_{N_{\ast}}$ and $C$, they must be functions of the moments $\mathbb{E}[x], \mathbb{E}[x^2], \ldots, \mathbb{E}[x^{2N_{\ast}}]$.
Second, if $d_{\rm tot} < d(N_{\ast}+1)$, all IPCs can reach unity.
Here, $d_{\rm{tot}}\in\N$ is the total number of polynomial terms used to represent the function of the input history.
To formalize the first limitation, we next introduce an algebraic characterization based on the elementary symmetric polynomials (Theorem~\ref{thm:average}) and their inversion via Vieta's formulas (Remark~\ref{rem:vieta}).
\begin{supptheorem}[Elementary symmetric polynomials and even-powers ensemble]\label{thm:average}
Let $r_1^2,\ldots,r_{N_\ast}^2$ be the time-invariant amplitudes appearing in Eq.~\eqref{eq:composition_form}. Then, the elementary symmetric polynomials $\{e_m\}_{m=1}^{N_\ast}$ of $r_1^2,\ldots,r_{N_\ast}^2$ can be expressed as functions of the even-order ensemble moments $\{\mathbb{E}[x^{2m}]\}_{m=1}^{N_\ast}$.
\end{supptheorem}
\begin{proof}
We denote by $e_1,\ldots,e_{N_\ast}$ the elementary symmetric polynomials of $r_1^2,\ldots,r_{N_\ast}^2$ (i.e., $e_1 = r_1^2 + \cdots + r_{N_\ast}^2$, $e_2 = \sum_{i<j} r_i^2 r_j^2$, $e_3 = \sum_{i<j<k} r_i^2 r_j^2 r_k^2$, \ldots).
First, as shown in Eqs.~\eqref{eq:composition_form} and \eqref{eq:ri_Prod}, $\mathbb{E}[x^n]$ is expressed as linear summations of finite products of $r_1,\ldots, r_{N_{\ast}}$.
Since the permutation of $r_1,\ldots,r_{N_\ast}$ does not affect Eq.~\eqref{eq:ri_Prod}; therefore $\mathbb{E}[x^n]$ is symmetric in these variables.
\par
Next, we prove that the symmetry of $\mathbb{E}[x^n]$ in Eq.~\eqref{eq:ri_Prod} can be reduced to symmetry with respect to $r_1^2,\ldots,r_{N_\ast}^2$.
The integrals in Eq.~\eqref{eq:ri_Prod}, vanish for odd exponents from Eq.~\eqref{eq:cos_prod_integral}.
This implies that the symmetry with respect to $r_1,\ldots, r_{N_{\ast}}$ is replaced by the symmetry with respect to $r^2_1,\ldots, r^2_{N_{\ast}}$.
Thus, $\mathbb{E}[x^n]$ can be expressed as a function of $e_1,\ldots,e_{N_{\ast}}$. 
\par
Finally, we prove that $e_{1},\ldots,e_{N_{\ast}}$ can be expressed as functions of $\{\mathbb{E}[x^{2m}]\}_{m=1}^{N_{\ast}}$.
From the definition of $e_1,\ldots,e_{N_{\ast}}$, the $e_l$ has polynomial order $2l$. Consequently, by matching polynomial orders, \(\mathbb{E}[x^{2m}]\) must take the form
\[
\mathbb{E}[x^{2m}] = d_m e_m + f(e_{m-1},\dots,e_1),
\]
where \(f\) is a polynomial in \(e_1,\dots,e_{m-1}\).
Considering the cases sequentially from $m = 1 $ to $ N_{\ast}$, one obtains the following expressions:
\begin{align*}
    e_1 &=d_1 \mathbb{E}[x^2],\\
    e_2 &=d_2 \mathbb{E}[x^4]-f_2(\mathbb{E}[x^2]),\\
    &\vdots\\
    e_{N_{\ast}} &= d_m\mathbb{E}[x^{2N_{\ast}}] - f_{N_{\ast}}(\mathbb{E}[x^2],\mathbb{E}[x^4],\ldots,\mathbb{E}[x^{2N_{\ast}-2}]),
\end{align*}
where $d_i$ is a constant (for $i=1\ldots N_{\ast}$).
Therefore, $e_1,\ldots,e_{N_{\ast}}$ are expressed as functions of $\{\mathbb{E}[x^{2m}]\}_{m=1}^{N_{\ast}}$.
\end{proof}
\begin{supprem}[Vieta's formulas~\cite{BlumSmithCoskey2017}]\label{rem:vieta}
Let $M \in \N$, and let $\alpha_1, \ldots, \alpha_M \in \R$.
Denote by $q_1, \ldots, q_M$ the elementary symmetric polynomials in $\alpha_1, \ldots, \alpha_M$.
Define the monic polynomial
$p(t) := t^M - q_1 t^{M-1} + q_2 t^{M-2} - \cdots + (-1)^M q_M$.
Then, $\alpha_1, \ldots, \alpha_M$ are exactly the roots of the equation $p(t)=0$, and hence
$p(t) = \prod_{i=1}^M (t - \alpha_i)$.
\end{supprem}
We now elaborate on these limitations. For the first limitation, we focus on the functional representation of $r^2_1,\ldots,r^2_{N_{\ast}}$ and $C$.
First, we derive a functional representation of the ensemble average as follows:
\begin{equation}
    \mathbb{E}[\phi(x_k)]=\mathcal{F}_{\phi}(r^2_1,\ldots,r^2_{N_{\ast}},C),
    \label{eq:limitation_1}
\end{equation}
where the observation function $\phi$ satisfies the assumptions of Theorem~\ref{thm:amsemble} and the map $\mathcal{F}_{\phi}:\mathbb{R}^{N_{\ast}+1}\to\mathbb{R}$ is determined by $\phi$.
From Theorem~\ref{thm:average} and Remark~\ref{rem:vieta}, there exists a function $G_{i}:\mathbb{R}^{N_{\ast}}\rightarrow \mathbb{R}$ such that
\begin{equation*}
    \begin{split}
        &r^2_i=G_{i}(\mathbb{E}[x_k^2],\ldots,\mathbb{E}[x_k^{2N_{\ast}}])\ (i=1,\ldots,N_{\ast}),\\
        &C=\mathbb{E}[x_k],
    \end{split}
\end{equation*}
which restricts the functional representations of $r^2_1,\ldots,r^2_{N_{\ast}}$ and $C$ to functions of only $\{\mathbb{E}[x^{2m}]\}_{m=1}^{N_{\ast}}$.
\par
The second limitation concerns the recovery of processed input.
In our theory, we focus on a single variable from the $d$-dimensional state space and on $(N_{\ast}+1)$ variables ($r_1^2,\ldots,r_{N_{\ast}}^2$ and $C$), which are represented by ensemble averages.
To extract all processed inputs for $d=1$, the condition $d_{\rm tot} < N_{\ast}+1$ must be satisfied, based on the relationship between the number of equations and unknowns.
So far, we have focused on a single direction; however, other variables are also available for extracting the processed inputs.
Using all $d$ dimensions, the number of equations increases, yielding $d(N_{\ast}+1)\;\geq\; d_{\rm tot}$.
\par
To illustrate these two limitations, we adopt the Stuart--Landau model in Eq.~{\rm (4)} as an example, for which corresponds to $N_{\ast}=1$.
First, when we focus only on the variable $x$ and apply the observation function $\phi=x^n$, the functional representation of $\mathbb{E}[\phi(x)]$ is described by
\begin{equation}
    \mathbb{E}[x^n]=
    \left\{
    \begin{array}{cc}
    \displaystyle \mathbb{E}[x] & (n\mbox{ is odd})\\
    \displaystyle \binom{n}{n/2}\left(\frac{\mathbb{E}[x^2]}{2}\right)^{\frac{n}{2}} & (n\mbox{ is even})
    \end{array}
    \right.,
    \label{eq:phi_SLsystem_modify}
\end{equation}
which is represented by a function of only $\mathbb{E}[x]$ and $\mathbb{E}[x^2]$.
Second, the functional representations of $r^2$ and $C$ become functions of only $\mathbb{E}[x]$ and $\mathbb{E}[x^2]$ as follows:
\begin{align*}
    r^2=\mathbb{E}[x^2],\ C=\mathbb{E}[x].
\end{align*}
\begin{figure*}
    \centering
    \includegraphics[width=\textwidth]{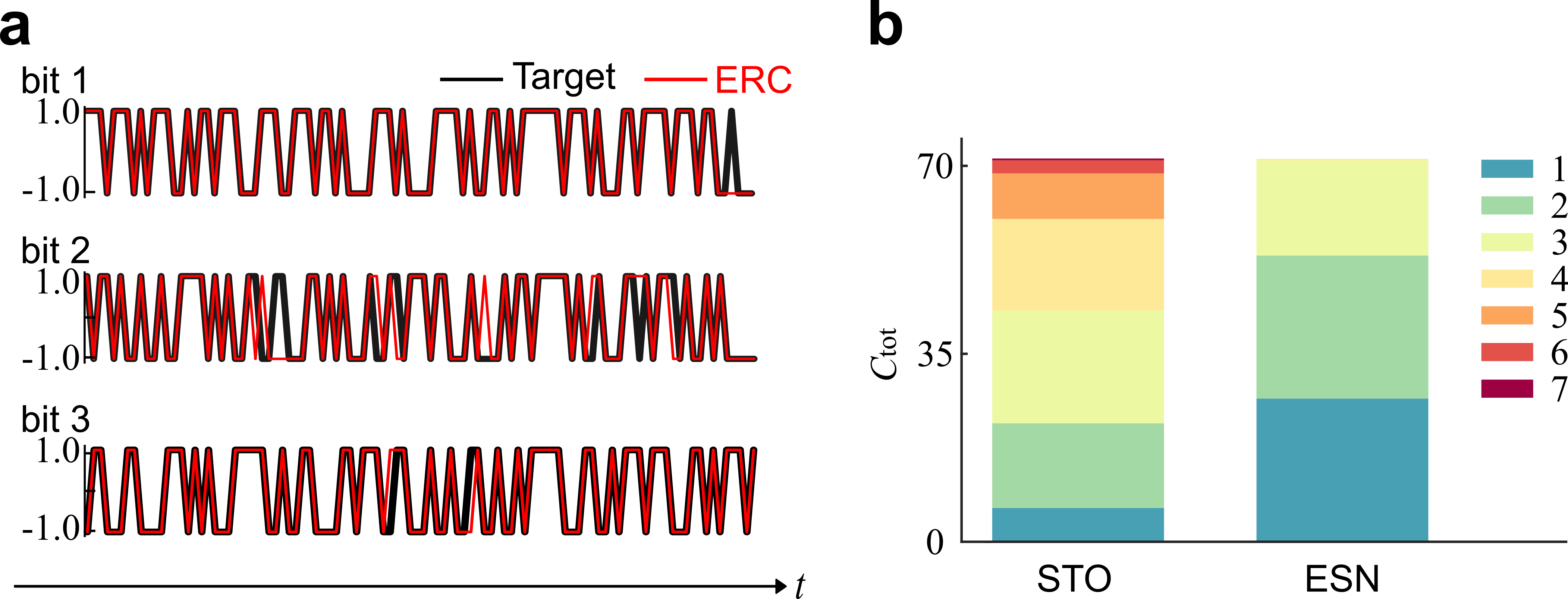}
    \caption{
    \textbf{CRC task performance and IPC using ERC with ADC.}
    \textbf{a}) 
    Time series comparison between the target (black) and ERC (red) signals of the CRC task.
    We use $\phi\in \{{\rm ADC}_2(x),{\rm ADC}_3(x),{\rm ADC}_4(x),{\rm ADC}_5(x)\}$ as nonlinear functions to solve the CRC task.
    \textbf{b}) IPC of the STO states used to solve the CRC task (left) and that of an ESN with comparable computational capability (right).
    }
    \label{fig:spin_task_with_ADC}
\end{figure*}
Finally, from the variables $r^2$ and $C$, we extract processed inputs (e.g. $u_{t-1},u_{t-2},u_{t-1}u_{t-2}$).
In this example, only one independent equation is available for the input history when $C=0$. Thus, the processed input can be identified only when $d_{\rm tot} = 1$.
However, all IPCs in the system may still be fully extracted by increasing the system dimension; for example, we can add a second variable of $y$ to the ensemble average. In addition, we can use the Stuart--Landau model with different parameters in Eq.~{\rm (4)}.
In this model, the input polynomial cannot be explicitly expressed in terms of the ensemble average. An illustrative example is instead provided using another model.
\par
We demonstrate the limitations of ERC using a $3$-dimensional system $\{x_i\}_{i=1}^{3}$.
The form of each direction is given as follows:
\[
x_i = \sum_{j=1}^{N_{\ast}}r_{i,j}(U)\cos(\theta_{i,j}+\eta_{i,j}(t,U))+C_{i}(U),
\]
where \(\eta_{i,j}\) depends on time and input history, \(\theta_{i,j}\) is uniformly distributed on \([-\pi,\pi]\), and \(r_{i,j}\) denotes the time-invariant quantities defined as follows:
\begin{equation}
    \begin{aligned}
        &r^2_{1,1}\coloneqq {u_{t-1}+u_{t-2}},\\
        &r^2_{1,2}\coloneqq {u^2_{t-1}+u^2_{t-2}+u_{t-2}},\\ &r^2_{1,3}\coloneqq u^2_{t-1}+u^2_{t-2},\ C_{1}\coloneqq 1,\\
        &r^2_{2,1}\coloneqq {u^2_{t-1}+u_{t-3}+u_{t-1}u_{t-2}},\\
        &r^2_{2,2}\coloneqq u_{t-2}+u_{t-3}+u_{t-4},\\
        &r^2_{2,3}\coloneqq u_{t-2}+u_{t-4}+u_{t-1}u_{t-2}+u^2_{t-1},\\
        &C_2 \coloneqq 1,\\
        &r^2_{3,1}\coloneqq {u_{t-1}u_{t-2}+u_{t-5}},\\
        &r^2_{3,2}\coloneqq u^2_{t-2}-u_{t-1}u_{t-2}-u_{t-5},\\
        &r^2_{3,3}\coloneqq u^2_{t-1}+u^2_{t-4}-u_{t-5},\ C_3 \coloneqq 1.
    \end{aligned}
\label{eq:invatiants}
\end{equation}
Then, $e_1,e_2$, and $e_3$ are given as follows:
\begin{equation}
    \begin{split}
        &e_1=2\left(\mathbb{E}[x^2]-\mathbb{E}[x]^2\right),\\
        &e_2=\frac{4}{3}\left[\mathbb{E}[x^4]-\left(\frac{3}{8}e_1^2+3e_1\mathbb{E}[x]^2+\mathbb{E}[x]^4\right)\right],\\
        &e_3=\left[\mathbb{E}[x^6]-\frac{1}{15}\left(\frac{5}{16}e_1^3+\right.\right.\\
        &~~~~~~~\frac{15}{8}e_1e_2+\frac{15}{8}e_1^2\mathbb{E}[x]^2+\frac{45}{8}e_1^2\mathbb{E}[x]^2+\\
        &~~~~~~~\left.\left.\frac{45}{4}e_2\mathbb{E}[x]^2+
        \frac{15}{2}e_1^2\mathbb{E}[x]^4+\mathbb{E}[x]^6\right)\right].
    \end{split}
\label{eq:symmetric_polynomials}
\end{equation}
The numbers of independent input variables in the $x_1$, $x_2$, and $x_3$ directions are $4$, $5$, and $5$, respectively, giving a total of $9$ distinct variables in the $3$-dimensional system.
We derive explicit forms of $\mathbb{E}[\phi(x_i)]$ (without the subscript~$i$) for $\phi\in\{x,x^2,\ldots,x^8\}$ as follows:
\begin{equation}
    \begin{split}
        &\mathbb{E}[x]=1,\ \mathbb{E}[x^2]=\frac{e_1}{2}+1,\ \mathbb{E}[x^3]=\frac{3}{2}e_1+1,\\
        &\mathbb{E}[x^4]=\frac{3}{8}e_1^2+2e_2+8e_1+1,\\
        &\mathbb{E}[x^5]=\frac{5}{8}(3e_1^2-2e_2+8e_1)+1,\\
        &\mathbb{E}[x^6]=\frac{5}{16}e_1^3+\frac{15}{8}e_1e_2+\frac{45}{8}e_1^2+\\
        &~~~~~~~~~~~\frac{45}{4}e_2+\frac{15}{2}e_1+1+15e_3,\\
        &\mathbb{E}[x^7]=\left[\frac{5}{16}e_1^3+\frac{305}{16}e_1e_2+\frac{225}{256}e_3\right]+\\
        &~~~~~~~~~~~\left[\frac{315}{8}e_1^2-\frac{105}{4}e_2\right]+
        \frac{21}{2}e_1+1,\\
        &\mathbb{E}[x^8]=
        \left[\frac{35}{2^7}(e_1^4-4e_1^2e_2+e_1e_3)+\frac{97}{2^7}\right]e_2+\\
        &~~~~~~~~~~~
        \left[\frac{35}{4}e_1^3-\frac{35}{2}(-e_1e_2+3e_3)\right]+\\
        &~~~~~~~~~~~
        \left[\frac{105}{4}e_1^2-\frac{105}{2}e_2\right]+
        \frac{e_1}{2}+1.
    \end{split}
     \label{eq:powere_ensemble}
\end{equation}
The quantities are symmetric with respect to $r^2_1,r^2_2$, and $r^2_3$ and are expressed as a function of $\mathbb{E}[x^6],\mathbb{E}[x^4],\mathbb{E}[x^2]$, and $\mathbb{E}[x]$.
Therefore, each $\mathbb{E}[\phi(x)]$ in Eq.~\eqref{eq:powere_ensemble} is a function of $\mathbb{E}[x^6]$, $\mathbb{E}[x^4]$, $\mathbb{E}[x^2]$, and $\mathbb{E}[x]$ for any $\phi\in\{x,x^2,\ldots,x^8\}$.
Consequently, values of $r_1^2$, $r_2^2$, and $r_3^2$ can be obtained as solutions of the following cubic equation:
\begin{align*}
    t^3-e_1t^2+e_2t-e_3=0\Leftrightarrow (t-r^2_{i,1})(t-r^2_{i,2})(t-r^2_{i,3})=0.
\end{align*}
Because the condition $d_{\mathrm{tot}} \leq N_\ast+1$, is not satisfied for a single state, complete recovery from one state is impossible.
However, partial reconstruction remains possible from subsets of the invariants.
\par
Typically, in the IPC framework, the processed input can be expanded in terms of orthogonal polynomials. Here, we present an example expanded using nonorthogonal polynomials:
\begin{align*}
    &u_{t-1}=r^2_{1,1}-r^2_{1,2}+r^2_{1,3},\ u_{t-2}=r^2_{1,2}-r^2_{1,3},\\
    &u_{t-3}=\frac{1}{2}r^2_{2,1}+\frac{1}{2}r^2_{2,2}-\frac{1}{2}r^2_{2,3},\\
    &u^2_{t-2}=r^2_{3,1}+r^2_{3,2}.
\end{align*}
From a single direction, only the above variables can be partially extracted. 
However, when further directions are taken into account, additional variables can be extracted as follows:
\begin{align*}
    &u_{t-4}=-r^2_{1,2}+r^2_{1,3}-\frac{1}{2}r^2_{2,1}+\frac{1}{2}r^2_{2,2}+\frac{1}{2}r^2_{2,3},\\
    &u^2_{t-1}=r^2_{1,3}-r^2_{3,1}-r^2_{3,2},\\
    &u^2_{t-4}=-\frac{1}{2}r^2_{2,1}+\frac{1}{2}r^2_{2,2}-\frac{1}{2}r^2_{2,3}+r^2_{3,1}+r^2_{3,3}.
\end{align*}
If we consider all three dimensions in Eq.~\eqref{eq:invatiants} simultaneously, the remaining input variables can be reconstructed as follows:
\begin{equation*}
    \begin{split}
        &u_{t-5}=r^2_{1,3}-\frac{1}{2}r^2_{2,1}+\frac{1}{2}r^2_{2,2}-\frac{1}{2}r^2_{2,3}-r^2_{3,2},\\
        &u_{t-1}u_{t-2}=-r^2_{1,3}+\frac{1}{2}r^2_{2,1}-\frac{1}{2}r^2_{2,2}+\frac{1}{2}r^2_{2,3}+r^2_{3,1}+r^2_{3,2},
    \end{split}
\end{equation*}
These results indicate that considering more directions progressively enables the extraction of additional variables, revealing the information embedded within the system.
\section{Comparison of an STO with an ESN}
We solved the CRC task using an experimental STO signal obtained from ERC.
Figure~\ref{fig:spin_task_with_ADC}a shows the performance of the CRC task with 2-, 3-, 4-, and 5-bit ADC($x$) as nonlinear functions.
The precision reached $94.8\%$, corresponding to the IPC of an ESN with $\rho \approx 0.6$ and $d =118$ (Fig.~\ref{fig:spin_task_with_ADC}b).
\section{TIPC computation procedure}
We provide the details of TIPC computation procedure in detail in Eq.~{\rm (15)}.
First, we scale the system such that the standard deviation of state $x$ is $1$.
The temporal average $\langle x\rangle$ and quadratic temporal average $\langle x^2\rangle$ are given as follows:
\begin{equation*}
    \begin{split}
        &\langle x\rangle=\frac{1}{2\pi}\int_{0}^{2\pi}x(s)ds=0,\ \langle x^2\rangle=\frac{1}{2\pi}\int_{0}^{2\pi}x^2(s)ds=\frac{11}{15}.
    \end{split}
\end{equation*}
We define the inner products of time $\langle\cdot,\cdot\rangle_{t}$ and input history $\langle\cdot,\cdot\rangle_{u}$ as:
\begin{equation*}
    \begin{split}
        &\langle f,g\rangle_t\coloneqq \frac{1}{2\pi}\int_{0}^{2\pi}f(s)g(s)ds,\ 
        \langle f,g\rangle_u\coloneqq \frac{1}{2}\int_{-1}^{1}f(s)g(s)ds.
    \end{split}
\end{equation*}
Second, we orthonormalize the orthogonal polynomial bases using the inner-products above:
\begin{equation*}
    \begin{split}
        &\psi_{\cos}\coloneqq \sqrt{6} P_1(u_{t-1})\cos t,\ \psi_{\sin}\coloneqq \sqrt{10} P_2(u_{t-2})\sin 2t,\\
        &\psi_{u_{\rm quad}}\coloneqq 3P_1(u_{t-1})P_1(u_{t-2}),\ \psi_{u_{\rm linear}}\coloneqq \sqrt{3}P_1(u_{t-1}),
    \end{split}
\end{equation*}
where $P_i$ denotes the $i$-th Legendre polynomials: $P_0(x)=1$, $P_1(x)=x$ and $P_2(x)=(3x^2-1)/2$.
We expand the scaled state $\tilde{x}\coloneqq (x-\langle x\rangle)/\sqrt{\langle x^2\rangle}$ using $\psi_{\cos}$, $\psi_{\sin}$, $\psi_{\rm linear}$, and $\psi_{\rm quad}$ as follows:
\begin{equation}
    \begin{split}
        \tilde{x}=&\sqrt{\frac{32}{45}}(u_{t-1}\cos{t}+u^2_{t-2}\sin{2t}+u_{t-1}u_{t-2}+u_{t-1})\\
        =&\left(\sqrt{\frac{15}{64}}\psi_{\cos}+\frac{1}{4}\psi_{\sin}+\right.\\
        &\left.\sqrt{\frac{5}{32}}\psi_{u_{\rm quad}}+\sqrt{\frac{15}{32}}\psi_{u_{\rm linear}}+\sqrt{\frac{5}{64}}\sin2t\right).
    \end{split}
    \label{eq:tipc_expansions}
\end{equation}
Figure~\ref{fig:TIPC_16} shows the IPC and TIPC obtained using Eq.~{\rm (15)}.
Each segment of the bar corresponds to a basis term used in the expansion in Eq.~\eqref{eq:tipc_expansions}, and its height equals the squared value of the associated expansion coefficient.
\begin{figure}[t]
    \includegraphics[width=0.48\textwidth]{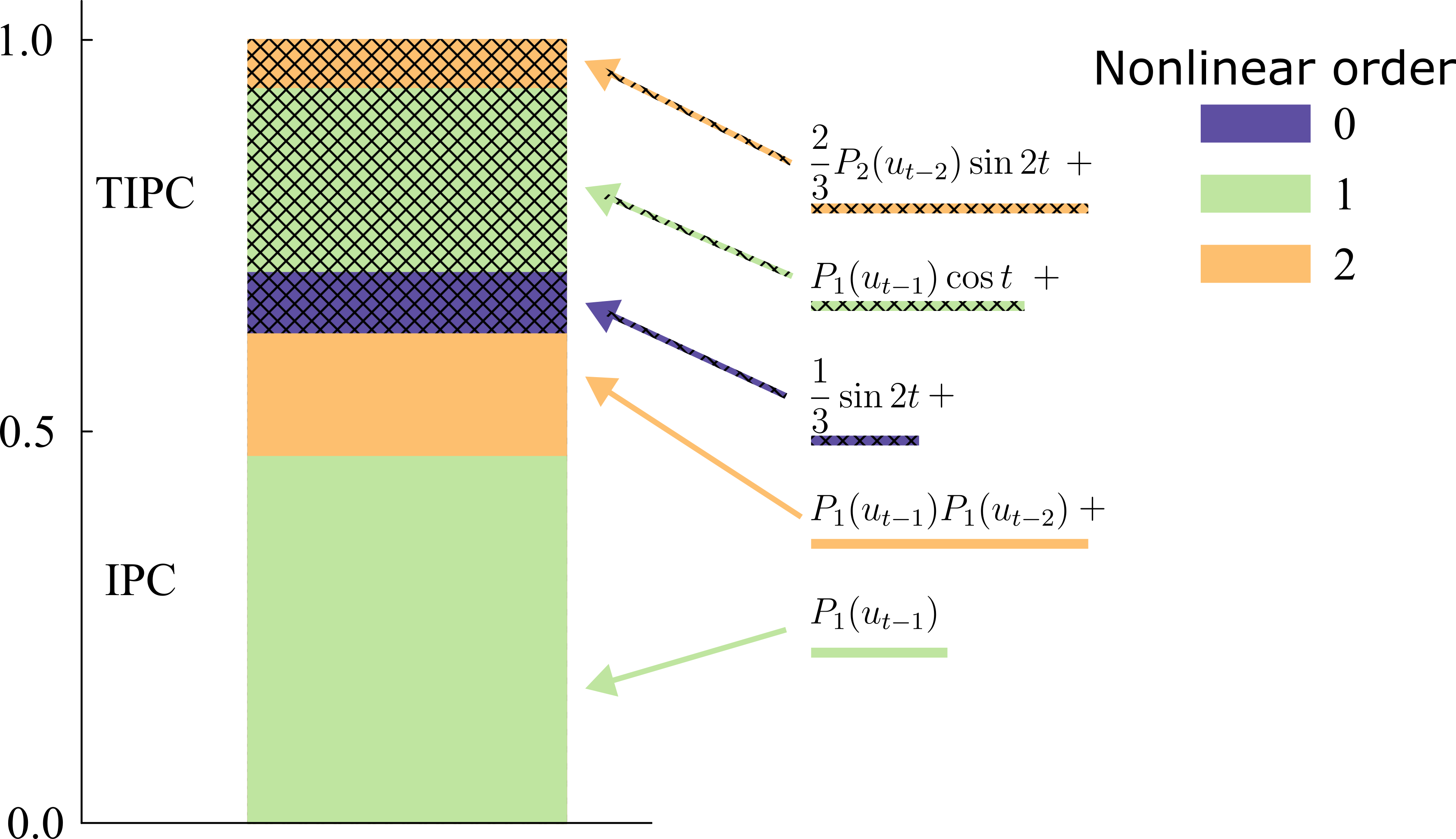}
    \caption{\textbf{Illustration of IPC and TIPC decomposition for the system in Eq.~\textbf{(15)}.}
    Bars show the total capacity with separation between the time-invariant (plain) and time-variant (hatched) components. Colors indicate the nonlinear input order: 0 (purple), 1 (green), and 2 (orange).}
    \label{fig:TIPC_16}
\end{figure}

\section{References}
\bibliographystyle{naturemag}
\bibliography{apssamp.bib}

\end{document}